\documentclass[11pt]{amsart}
\usepackage{fullpage,amsmath,amsthm,amsfonts,amssymb,
graphicx,amscd,float,hyperref,enumitem,setspace}
\usepackage{comment}
\usepackage{graphicx,xypic}

\newtheorem{thm}{Theorem}[section]

\newtheorem{theor}{Theorem} 

\newtheorem{theore}{Theorem} 

\newtheorem{corol}{Corollary} 

\newtheorem{coroll}{Corollary}

\newtheorem{lem}[thm]{Lemma}

\newtheorem{cor}[thm]{Corollary}
\newtheorem{corollary}[thm]{Corollary}
\newtheorem{prop}[thm]{Proposition}

\newtheorem{cl}[thm]{Claim}
\theoremstyle{definition}
\newtheorem{rk}[thm]{Remark}
\newtheorem{conv}[thm]{Convention}
\newtheorem{df}[thm]{Definition}
\newtheorem{pd}[thm]{Proposition-Definition}

\newtheorem*{claim}{Claim}

  \newcommand{\F}{\mathbb{F}}

  \newcommand{\N}{\mathbb{N}}
  
  \newcommand{\PP}{\mathbb{P}}
  
  \newcommand{\R}{\mathbb{R}}

  \newcommand{\Z}{\mathbb{Z}}

\newcommand{\iwp}{IW(\varphi)}

\newcommand{\vphi}{\varphi}
\newcommand{\veps}{\varepsilon}
\newcommand{\ol}{\overline}
\newcommand{\from}{\colon}

\newcommand{\out}{\textup{Out}(F_r)}
\newcommand{\diam}{\textup{diam}}
\newcommand{\norm}[1]{|#1|}
\DeclareMathOperator{\Out}{Out}

\newcommand{\mL}{\gamma}

\newcommand{\G}{\Gamma}
\newcommand{\uos}{\overline{\rm{CV}}}
\newcommand{\os}{\rm{CV}}
\newcommand{\cv}{\os}

\newcommand{\AT}{\mathcal{AT}}
\newcommand{\ue}{\mathcal{UE}}
\newcommand{\FF}{\mathcal{FF}}
\newcommand{\val}{\rm{val}}

\newcommand{\ds}{d_{\mathrm{sym}}}
\newcommand{\rmu}{\check \mu}
\newcommand{\rnu}{\check \nu}
\newcommand{\dL}{d_{\rm{CV}}}
\newcommand{\smgp}{\langle \rm{supp}(\mu)\rangle_+}
\newcommand{\Mod}{\mathrm{Mod}}

\newcommand{\gp}[3]{(#2 \vert #3)_{#1}} 

\newcommand{\ind}{\mbox{ind}}

\begin{document}

\title[Random outer automorphisms of free groups]{Random outer automorphisms of free groups:\\ Attracting trees and their singularity structures}

\author{Ilya Kapovich, Joseph Maher, Catherine Pfaff, and Samuel J. Taylor}

\address{Department of Mathematics, University of Illinois at Urbana-Champaign\newline
  \indent 1409 West Green Street, Urbana, IL 61801
  \newline \indent  {\url{https://faculty.math.illinois.edu/~kapovich}}, }
  \email{\tt kapovich@math.uiuc.edu}

\address{ CUNY College of Staten Island and CUNY Graduate Center \newline
  \indent
2800 Victory Boulevard, Staten Island, NY 10314
  \newline
  \indent  {\url{http://www.math.csi.cuny.edu/~maher/}}, }
  \email{\tt joseph.maher@csi.cuny.edu}

\address{Department of Math \& Stats, Queen's University \newline
  \indent
Jeffery Hall, 48 University Ave.,
Kingston, ON Canada, K7L 3N6
  \newline
  \indent  {\url{http://math.ucsb.edu/~cpfaff/}}, }
  \email{\tt cpfaff@math.ucsb.edu}

\address{Department of Mathematics, Temple University \newline
  \indent
1805 Broad St, Philadelphia, PA 19122
  \newline
  \indent  {\url{https://math.temple.edu/~samuel.taylor/}}, }
  \email{\tt samuel.taylor@temple.edu}

  \dedicatory{Catherine Pfaff would like to dedicate this paper to her father Dr. Roland Pfaff, zikhrono livrakha.}

  \thanks{The first named author was supported by the individual NSF
    grants DMS-1405146 and DMS-1710868. The second named author was
    supported by Simons Foundation and PSC-CUNY.
    The fourth named author was partially supported by NSF grant
    DMS-1744551. All authors acknowledge support from U.S. National
    Science Foundation grants DMS 1107452, 1107263, 1107367 ``RNMS:
    GEometric structures And Representation varieties'' (the GEAR
    Network).}

\subjclass[2010]{Primary 20F65, Secondary 57M, 37B, 37D}

\begin{abstract}
We prove that a ``random" free group outer automorphism is an ageometric fully irreducible outer automorphism
whose ideal Whitehead graph is a union of triangles.
In particular, we show that its attracting (and repelling) tree is a nongeometric $\R$--tree all of whose branch points are trivalent.
\end{abstract}

\keywords{Free group, random walk, Outer space, free group automorphisms, train track maps}

\maketitle

\section{Introduction}

In recent years advances in the theory of random walks have produced significantly better understanding of algebraic and geometric properties of ``random" 
elements in various groups.  For example, a random walk on the mapping class group of a surface gives rise to a pseudo-Anosov element with asymptotic probability one, as shown by Rivin \cite{r08} and Maher \cite{m11}, and a random walk on $\out$ gives rise to a fully irreducible element with asymptotic probability one, as shown by Rivin \cite{r08}. However, being pseudo-Anosov for elements of $\Mod(\Sigma)$ or being fully irreducible for elements of $\out$ are only the basic classification properties of elements of these groups. Pseudo-Anosovs and fully irreducibles have much more interesting further structural stratification coming from their ``index" or ``singularity'' properties.

Recently, Gadre and Maher provided such finer information in the mapping class group setting \cite{gm16}. They showed that for a random walk $(w_n)$ on $\Mod(\Sigma)$ defined by a measure $\mu$
(under appropriate natural assumptions on $\mu$ and $\Sigma$)
the mapping class $w_n$ is a pseudo-Anosov whose axis lies in the \emph{principal stratum}
with probability going to $1$ as $n \to \infty$.
Roughly, the unit cotangent bundle to $\mathcal T(\Sigma)$ is naturally identified with the space of unit area holomorphic quadratic differentials on $\Sigma$. For a quadratic differential to be \emph{principal}, or to be in the \emph{principal stratum}, means that the differential has simple zeros (and simple poles at punctures).
The axis of a pseudo-Anosov $\varphi$ is principal if its associated quadratic differentials are principal.
If we assume that $\Sigma$ is closed, this means that
the dual $\mathbb R$-trees (see \cite[Ch 11]{k09}) $T_+(\varphi)$ and $T_-(\varphi)$ associated to the stable and unstable invariant foliations are \emph{trivalent}. To be precise, for a point $p$ in an $\mathbb R$-tree $T$, the valency $\val_T(p)$ of $p$ in $T$ is the number of connected components of $T\setminus\{p\}$. A point $p\in T$ is called a \emph{branch point} if $\val_T(p)\ge 3$, and an $\mathbb R$-tree $T$ is trivalent if for every branch point $p\in T$ we have $\val_T(p)=3$.
Note that, for topological reasons, the number of (three-prong) singularities of the stable/unstable foliation of $w_n$ (which is the same as the number of $\Mod(\Sigma)$-orbits of branch-points in the corresponding dual $\mathbb R$-tree) is then fixed and depends only on the topology of the surface $\Sigma$.

In this paper we study similar questions for ``random'' elements of $\out$, the outer automorphism group of the free group $F_r$, where $r\ge 3$. Recall that an element $\vphi\in\out$ is called \emph{fully irreducible} if no positive power of $\vphi$ leaves invariant a nontrivial proper free factor of $F_r$. Fully irreducibles serve as the main counterparts to pseudo-Anosovs in the $\out$ setting and play a crucial role in the study of $\out$. The index theory of $\out$ goes back to the 1990s~\cite{gl95,gjll} and turns out to be considerably more complicated than in the mapping class group setting. For simplicity assume that $\vphi\in \out$ is fully irreducible. Then the original definition of the \emph{index} of $\vphi$ involves the dynamics of various representatives of $\vphi$ in ${\rm Aut}(F_r)$ when acting on $\partial F_r$, and specifically the structure of periodic points.

One of the key differences between the mapping class group case and the $\out$ case is that the value of $i(\vphi)$ is not fixed as a function of the rank $r$ and in general one only has the inequality $0>i(\vphi)\ge 1-r$. The values of $i(\vphi)$ actually help us classify outer automorphisms. For so-called \emph{geometric} fully irreducibles $\vphi\in \out$, that is, those that are induced by pseudo-Anosov homeomorphisms of compact surfaces with one boundary component and with fundamental group isomorphic to $F_r$, one has $i(\vphi)=1-r$.  Nongeometric fully irreducible outer automorphisms come in two flavors: \emph{parageometric} ones, for which $i(\vphi)=1-r$, and \emph{ageometric} ones, for which $i(\vphi)>1-r$.

Another invariant of a fully irreducible is its associated \emph{attracting tree} $T_+^{\vphi}$ in the boundary $\partial \cv$ of the Culler--Vogtmann Outer space $\cv$. These two invariants, namely $i(\vphi)$ and $T_+^{\vphi}$, are interconnected. For a nongeometric fully irreducible $\vphi$, we have that $T_+^{\vphi}$ is an $\mathbb R$-tree endowed with a free isometric action of $F_r$ with dense orbits, and with only finitely many $F_r$-orbits of branch points~\cite{gl95}. One can interpret the index properties of $\vphi$ in terms of this action. In particular, $i(\vphi)$ is equal to $\sum \left(1-\frac{\val(p)}{2}\right)$ taken over all representatives $p$ of $F_r$-orbits of branch points in $T_+^{\vphi}$.  It is also known that $i(\vphi)=1-r$ (i.e. $\vphi$ is geometric or parageometric) if and only if the tree $T_+^{\vphi}$ is ``geometric" in the sense of~\cite{LevittPaulin}, that is, $T_+^{\vphi}$ is dual to a foliation on some finite 2-complex.

The index theory of $\out$ contains some additional, finer data. First, for a nongeometric fully irreducible $\vphi\in\out$, it records the ``index list" of $\vphi$, which we can think of as the list of valencies/degrees of $F_r$-representatives of branch points of $T_+^{\vphi}$. Second, at each branch point $p$ of $T_+^{\vphi}$, apart from the degree of $p$ in $T_+^{\vphi}$, one can also record the ``turns" taken by the attracting lamination $\Lambda_{\vphi}$ at $p$ in $T_+^{\vphi}$. The union of this data over representatives of all $F_r$-orbits of branch points in $T_+^{\vphi}$ is called the \emph{ideal Whitehead Graph} $\iwp$, which was introduced by Handel and Mosher in \cite{hm11}. As with the index, in practice the ideal Whitehead graph $\iwp$ can be read-off using a train track representative of $\vphi$ and it is often defined in such a way. See \S \ref{ss:Trees} below for additional background, definitions, and references regarding the index theory for $\out$.

\subsection{Main results}
Before stating our main results, we introduce a few more definitions (even further notation and explanations follow the statement). For $r\ge 3$, we denote by $\Delta_r$ the graph which is the disjoint union of  $2r-3$ triangles. We say that an ageometric fully irreducible $\vphi\in\out$ is \emph{principal} if $\iwp\cong \Delta_r$. We say that an ageometric fully irreducible $\vphi\in\out$ is \emph{triangular} if $\iwp$ is isomorphic to the union of several components of $\Delta_r$, that is, if $\iwp$ is the disjoint union of $\le 2r-3$ triangles. Note that for a principal $\vphi$ we have $i(\vphi)=\frac{3}{2}-r$ and for a triangular $\vphi$ we have $i(\vphi)\ge \frac{3}{2}-r$.  We can now state our main result:

\begin{theor}\label{T:main}
Let $r\ge 3$ and let $\mu$ be a probability distribution on
$\Out(F_r)$ which is nonelementary and has bounded support for the action on $\FF$, and let
$(w_n)$ be the random walk determined by $\mu$.  Suppose
that $\smgp$ contains $\vphi^{-1}$ for some principal fully irreducible $\varphi\in\out$. Then $w_n$ is a triangular fully irreducible outer automorphism with probability going to $1$ as $n \to \infty$.
\end{theor}

\stepcounter{theor}

Here $\FF$ is the \emph{free factor graph} for $F_r$, which is the free group analog of the curve graph (see \S \ref{sec:ff}). The main point to remember for now is that $\FF$ is Gromov-hyperbolic and that $\psi\in\out$ acts loxodromically on $\FF$ if and only if $\psi$ is fully irreducible, as shown by Bestvina and Feighn \cite{bf11}. Above, $\smgp$ denotes the subsemigroup of $G$ generated by the support of $\mu$. The action on $\FF$ is considered \emph{nonelementary} if $\smgp$ contains two independent loxodromic isometries of $\FF$.


Theorem~\ref{T:main} has several important consequences:

\begin{corol}\label{C:rw}
Let $r\ge 3$ and let $\mu$ be a probability distribution on $\Out(F_r)$ which is nonelementary and has bounded support for the action on $\FF$, and let $(w_n)$ be the random walk determined by $\mu$. Suppose that $\smgp$ contains $\vphi^{-1}$ for some principal fully irreducible $\vphi\in\out$.

Then, with probability going to $1$ as $n \to \infty$, $w_n$ is an ageometric fully irreducible outer automorphism $\vphi_n$ satisfying:
\begin{enumerate}
\item each component of the ideal Whitehead graph $\iwp$ is a triangle;
\item each branchpoint of $T_+^{\vphi_n}$ is trivalent;
\item $i(\vphi_n)>1-r$ and $ind_{geom} (T_+^{\vphi_n})<2r-2$ (here $i(\vphi_n)$ is the rotationless index of $\vphi_n$ and $ind_{geom} (T_+^{\vphi_n})$ is the geometric index of $\vphi_n$);
\item the tree $T_+^{\vphi_n}$ is nongeometric.
\end{enumerate}
\end{corol}

\stepcounter{theor}

In fact, our arguments apply equally well to the inverse of a random outer automorphism:
\begin{corol}\label{C:two}
Let $r\ge 3$ and let $\mu$ be a probability distribution on
$\Out(F_r)$ which is nonelementary and has bounded support for the action on $\FF$, and let
$(w_n)$ be the random walk determined by $\mu$.  Suppose
that $\smgp$ contains a pair
of elements $\vphi$ and $\psi$, such that both $\vphi$ and $\psi^{-1}$
are principal fully irreducible outer automorphisms.

Then, with probability approaching $1$ as
$n \to \infty$, both $w_n$ and $w_n^{-1}$ are triangular fully
irreducible outer automorphisms. In particular, both fixed trees of $w_n$ are trivalent.
\end{corol}

We remark that
it is proved in \cite[Example 6.1]{stablestrata} that for each $r\ge 3$ there exists a principal $\vphi\in\out$.  Thus the assumption in Theorem~\ref{T:main} and in Corollary~\ref{C:rw} that $\smgp$ contain the inverse of some principal outer automorphism is automatically satisfied if, for example, $\smgp=\out$ or if $\smgp$ is a subgroup of finite index in $\out$. Similarly, the assumption in Corollary~\ref{C:two}, that $\smgp$ contains a pair
of elements $\vphi$ and $\psi$ such that both $\vphi$ and $\psi^{-1}$
are principal fully irreducible outer automorphisms, is also automatically satisfied in the case where either $\smgp=\out$ or $\smgp$ is a subgroup of finite index in $\out$.

\subsection{Key ingredients}
The proof of Theorem~\ref{T:main} involves several key ingredients. The first is the ``stability" result of Algom-Kfir, Kapovich, and Pfaff~\cite{stablestrata} about the axes of principal outer automorphisms in Outer space $\cv$. The second is a strengthened
contraction property for axes of fully irreducibles in $\cv$, which we call the bounded geodesic image property; see Theorem~\ref{t:BGIP}. This is stronger than previous contraction properties that were established in \cite{a08} and later in \cite{bf11,dt17}.
The third ingredient is the \emph{lone axis} property for principal outer automorphisms. Lone axis fully irreducibles $\vphi\in \out$ are those fully irreducibles for which the entire ``axis bundle" (introduced by Handel and Mosher in \cite{hm11}) consists of a single line, and consequently, which have a unique axis in $\cv$. Lone axis fully irreducibles were introduced and studied by Mosher and Pfaff in \cite{loneaxes}, where also an explicit characterization (in terms of the index theory) of the lone axis property was obtained. The results of \cite{loneaxes}, as well of ~\cite{ap16,stablestrata}, demonstrate the existence of deep connections between the index properties of fully irreducibles, particularly the structure of their ideal Whitehead graphs, and the geometric properties of their axes in $\cv$. In the present  paper we use the lone axis property of principal outer automorphisms to ``coarsify'' the infinitesimal stability results from \cite{stablestrata}, see Corollary~\ref{cor:close_enough} below.

\subsection{Further questions}


Theorem~\ref{T:main} naturally raises the questions of whether the properties of being principal or lone axis are typical, in the random walk sense, for the elements of $\out$. We believe that the answer to both of these questions is probably negative, but that principal outer automorphisms do occur with asymptotically positive probability. Note that, as a consequence of the results of \cite{loneaxes}, for a triangular automorphism $\vphi$ it follows that $\vphi$ is lone axis if and only if $\vphi$ is principal, i.e. if and only if $i(\vphi)=\frac{3}{2}-r$. Computer experiments, that we conducted using Thierry Coulbois' SAGE train track package\footnote{The Coulbois computer package is available at \cite{c12}.}, appear to indicate that several possibilities for $i(\vphi)$ (and not just $i(\vphi)=\frac{3}{2}-r$) happen with asymptotically positive probability. Moreover, specific examples exhibited in \cite{stablestrata} demonstrate a specific reason why a loss of several components of $\iwp$ may occur, resulting in the loss of the lone axis property. Roughly, a periodic vertex in a train track representative of $\vphi$, which contributes to $\iwp$, may become nonperiodic for a ``nearby" outer automorphism $\psi$ and thus no longer contribute to $IW(\psi)$. Heuristic considerations suggest that this behavior will have detectable probabilistic effect.

However, this assumption in Theorem~\ref{T:main} cannot, in general, be removed. In ~\cite{kp15} Kapovich and Pfaff constructed a ``train track directed'' random walk on $\out$. While that walk is defined by an irreducible finite state Markov chain, rather than by a sequence of i.i.d. random variables distributed according to $\mu$, qualitatively the walk behaves similarly to the ones considered in the present paper. However, with asymptotically positive probability, for the outer automorphisms $\vphi_n$ produced by that walk after $n$ steps, the ideal Whitehead graph $IW(\vphi_n)$  is the complete graph on $2r-1$ vertices.  Thus $\vphi_n$ is neither principal nor triangular in that case. The walk considered in \cite{kp15} is constructed in such a way that it never encounters principal outer automorphisms.

\subsection*{Acknowledgements}

We are grateful to Lee Mosher and Chris Leininger for useful conversations.
We also thank Yael Algom-Kfir, who contributed to early discussions of the problems considered in this paper.

\vskip20pt

\section{Metric spaces, random walks, and oriented matches}

\subsection{Metric space notation}\label{ss:msnotation}

In this section we collect together some basic definitions about (possibly asymmetric) metric spaces.  The main two examples we shall consider are Outer space with the Lipschitz metric $(\os, \dL)$, an asymmetric geodesic metric space which is locally compact, and the complex of free factors $(\FF, d_\FF)$, a symmetric Gromov hyperbolic space which is not locally compact.  See \S \ref{ss:os} and \S \ref{sec:ff} for further details.

An (asymmetric) metric space is a pair $(X, d_X)$, where $X$ is a topological space, and $d_X$ is a function $d_X \colon X \times X \to \R_{\ge 0}$ which satisfies $d_X(x, x) = 0$ for all points $x \in X$, and $d_X(x, y) \le d_X(x,z) +d_X(z,y)$, for all points $x, y$, $z \in X$.  We say that $(X, d_X)$ is \emph{symmetric} if $d_X(x, y) = d_X(y, x)$ for all $x,y\in X$.

We say that the (possibly asymmetric) metric space $(X, d_X)$ is \emph{geodesic} if for each pair of points $x,y\in X$ there is a subinterval $I = [0, d_X(x, y)] \subset \R$ and a path $\gamma \colon I \to X$ such that $d_X(x, \gamma(t) ) = t$.  Let $I$ be a connected subset of $\R$, or a set of consecutive integers in $\Z$, and let $Q$ be a non-negative constant.  We say that a map $\gamma \colon I \to (X, d_X)$ is a \emph{$Q$-quasigeodesic} if for all $s < t$ in $I$,
\[ \frac{1}{Q} \norm{s - t} - Q \le d_X(\gamma(s), \gamma(t)) \le Q
\norm{s - t} + Q. \]

We say that $X$ is \emph{Gromov hyperbolic}, or \emph{$\delta$-hyperbolic}, if $(X, d_X)$ is a symmetric geodesic metric space, and there is a constant $\delta \ge 0$ such that for any geodesic triangle in $X$, each side is contained in a $\delta$-neighbourhood of the other two.

The \emph{Gromov product} of three points $x, y,z \in X$ is defined to be
\[ \gp{x}{y}{z} = \tfrac{1}{2}( d_X(x, y) + d_X(x, z) - d_X(y,
z)).  \]
If $X$ is $\delta$-hyperbolic, then $\gp{x}{y}{z}$ is equal to the distance from $x$ to a geodesic from $y$ to $z$, up to an error of at most $2 \delta$. Given two points $x$ and $y$ in $X$, and a number $R$, we define the \emph{shadow} $S_x(y, R)$ to be
\[ S_x(y, R) = \{ z \in X \mid \gp{x}{y}{z} \ge d_X(x, y) - R \}.  \]

The \emph{Gromov boundary $\partial X$} may be defined to be equivalence classes of quasigeodesic rays $\gamma \colon \R_{\ge 0} \to X$, where two quasigeodesic rays $\gamma$ and $\gamma'$ are equivalent if there exists a value $K$ such that $d(\gamma(t),\gamma'(t))<K$ for all $t\in\R_{\ge 0}$. Given $\delta \ge 0$, there is a constant $Q \ge 0$ such that any two distinct points in $\partial X$ are connected by a bi-infinite $Q$-quasigeodesic.

Let $g$ be an isometry of $X$.  We say that $g$ is \emph{loxodromic} if the orbit of $( g^n x)_{n \in \Z}$ of any point $x \in X$ is a quasigeodesic in $X$.  A loxodromic element $g$ has precisely two fixed points in the boundary, and we call any bi-infinite $Q$-quasigeodesic connecting them a \emph{quasi-axis} for $g$, though we may always assume the quasi-axis is $g$-invariant.

\subsection{Random walks background}

Let $G$ be a countable group, and let $\mu$ be a probability distribution on $G$.  We shall define the \emph{step space} to be the countable product $(G, \mu)^\Z$, and we shall write $g_n$ for projection onto the $n$-th factor.  The random variables $(g_n)_{n \in \Z}$ then correspond to a bi-infinite sequence of independent $\mu$-distributed random variables.  A sequence of steps $\omega \in (G, \mu)^\Z$ determines a random walk, whose location at time $n$ is given by the product of the first $n$ steps.  More precisely, the location of the random walk at time $0$ is $w_0 = 1$, and the location of the random walk at time $n$ is given by $w_{n+1} = w_n g_n$.  We shall write $\check \mu$ for the reflected measure, i.e. $\check \mu(g) = \mu(g^{-1})$, and we shall write $\mu^{*n}$ for the $n$-fold convolution of $\mu$ with itself.  In particular, $w_n$ is a random variable on $(G, \mu)^\Z$ with values in $G$, whose distribution is given by $\mu^{*n}$ if $n > 0$  and by $\check \mu^{*n}$ if $n < 0$.

For a group $G$ with a probability distribution $\mu$, the \emph{support} $\rm{supp}(\mu)$ of $\mu$ is the set of all elements $g\in G$ such that $\mu(g)>0$. We also denote by $\smgp$ the sub-semigroup of $G$ generated by $\rm{supp}(\mu)$. Note that for $g\in G$ we have $g\in \smgp$ if and only if there exists an $n\ge 1$ such that $\mu^{*n}(g)>0$.

Let $G$ act on a hyperbolic metric space $X$ by isometries.  We say two loxodromic isometries are \emph{independent} if their fixed point sets are disjoint in the Gromov boundary $\partial X$.  We say that the probability distribution $\mu$ is \emph{nonelementary} if $\smgp$ contains two independent loxodromic elements.

A choice of basepoint $x_0 \in X$ determines an orbit map from $G$ to $X$ given by $g \mapsto g x_0$.  A random walk $(w_n)_{n \in \Z}$ on $G$ then gives rise to a sequence $(w_n x_0)_{n \in \Z}$ in $X$.  If $\mu$ is nonelementary then almost every sequence $(w_n)_{n \in \Z}$ converges to a point in $\partial X$ in both the forward and backward directions.  Let $\gamma_\omega$ be a bi-infinite $Q$-quasigeodesic in $X$ connecting the two limit points.  We say that $\mu$ has \emph{bounded support in $X$} if the image of the support of $\mu$ in $X$ under the orbit map has bounded diameter.
If $\mu$ has bounded support in $X$, then there are constants $\ell > 0$, $K \ge 0$, and $c < 1$ such that the following estimates hold.  The random walk has positive drift \cite[Theorem 1.2]{MaherTiozzo}, i.e. for all $n \ge 0$,
\begin{equation}\label{eq:drift}
\PP( d_X(x_0, w_nx_0) \le \ell n ) \le Kc^n.
\end{equation}
Furthermore, there is exponential decay for the convolution measures of shadows \cite[Lemma 2.10]{maher12}, i.e. for all $n$, for all $y \in X$, and for all $R \ge 0$,
\begin{equation}\label{eq:shadows}
\mu^{*n}( S_{x_0}(y, R) ) \le K c^{d_X(x_0 , y) - R}.
\end{equation}
The following two estimates follow from exponential decay for shadows.  For all $n \in \Z$ and for all $r \ge 0$, we have the following estimate for the probability of a location of the sample path $w_n x_0$ being far from $\gamma_\omega$ \cite[Proposition 5.7]{MaherTiozzo}:
\begin{equation}\label{eq:dist}
\PP( d_X( w_n x_0 , \gamma_\omega) \ge r ) \le K c^r.
\end{equation}
Finally, for Gromov products we have the following estimate for all $0 \le i \le n$ and all $r \ge 0$:
\begin{equation}\label{eq:gp}
\PP( \gp{w_i x_0}{x_0}{w_n x_0} \ge r ) \le K c^r.
\end{equation}
A weaker form of this estimate is stated in \cite[Proposition 5.9]{MaherTiozzo}, but the proof essentially shows the version above.  We give the argument below for the convenience of the reader.

\begin{proof}[Proof of \eqref{eq:gp}]
If $\gp{w_i x_0}{x_0}{w_n x_0} \ge r$, then $x$ lies in a shadow $S_{w_i x_0}(w_n x_0, R)$, with $d(w_i x_0, w_n x_0) - R \ge r + O(\delta)$.  The random variables $w_i$ and $w_i^{-1} w_n$ are independent, so by exponential decay of shadows \eqref{eq:shadows}, this occurs with probability at most $K c^{r + O(\delta)}$.
\end{proof}

We remark that \emph{a priori} the constants $K$ and $c$ in \eqref{eq:drift}--\eqref{eq:gp} above may all be different, but without loss of generality, we may choose $K$ and $c$ to be the largest of any of these constants, and then \eqref{eq:drift}--\eqref{eq:gp} hold for the same $K$ and $c$.

\subsection{Oriented matches}


Given $\kappa >0$, a quasigeodesic segment $\gamma \colon J \to X$, and a quasigeodesic $\gamma' \colon I \to X$, we follow [DH] and say that \emph{$\gamma'$ crosses $\gamma$ up to distance $\kappa$} if there exists an increasing map $\theta \colon J \to I$ such that $d_X(\gamma(t),\gamma'(\theta(t))) \le \kappa$ for all $t \in J$.

\begin{prop}\cite[Proposition 2.5]{DahmaniHorbez}\label{prop:dh}
Given non-negative constants $\delta$ and $Q$, there is a constant $\kappa > 0$ such that if $G$ is a group acting by isometries on a $\delta$-hyperbolic space $X$, and if $\mu$ is a nonelementary probability measure on $G$ with bounded support in $X$, then for almost every sample path $\omega \in (G , \mu)^\Z$ of the random walk, and all $0<\veps<1$, there exists an integer $N\ge 0$ such that for all integers $n\ge N$, the element $w_n$ acts loxodromically on $X$, and for all geodesic segments $\gamma_n$ from $x_0$ to $w_n x_0$, every $Q$--quasiaxis $\alpha_{w_n}$ for $w_n$ crosses a subsegment of $\gamma_n$ of length at least $(1- \veps)d_X(x_0,w_nx_0)$ up to distance $\kappa$.
\end{prop}

This result is stated for $\mu$ with finite support in \cite[Proposition 2.5]{DahmaniHorbez}, but their proof only uses finite support to give positive drift with exponential decay, and exponential decay for the measures of shadows, so their result holds for $\mu$ with bounded support in $X$.

Now suppose we have an isometric action $G \curvearrowright X$ and let $\gamma$ and $\gamma'$ be quasigeodesics in $X$. We say that $\gamma$ and $\gamma'$ have an $(L,\kappa)$--\emph{oriented match} if there is a subpath $s \subset \gamma$ of diameter at least $L$ and some $h\in G$ such that $h \cdot \gamma'$ crosses $s$ up to distance $\kappa$.


The following proposition is essentially given in \cite[Proposition 3.2]{MaherSisto}.  However, there it is assumed that the action is acylindrical and the authors were not concerned with the orientation of the match.

\begin{prop}\label{prop:MS}
Given non-negative constants $\delta$ and $Q$, there is a constant $\kappa_0$, such that if $G \curvearrowright X$ is a nonelementary action on a $\delta$-hyperbolic space, $\mu$ is a nonelementary measure on $G$ for this action, and $g \in G$ is a loxodromic in the semi-group generated by the support of $\mu$, then for all $L \ge 0$ and $\kappa \ge \kappa_0$, the probability that a $Q$--quasiaxis $\alpha_{w_n}$ for $w_n$ has an $(L, \kappa)$--oriented match with a $Q$--quasiaxis $\alpha_g$ of $g$ goes to $1$ as $n \to \infty$.
\end{prop}

Proposition \ref{prop:MS} follows from Proposition \ref{prop:dh} and \cite[Proposition 3.2.4]{MaherSisto}.  Propositions 3.2.1--3.2.5 in \cite{MaherSisto} are stated for acylindrical actions, but acylindricality is only used in Propositions 3.2.1 and 3.2.2.  Proposition \ref{prop:MS} follows from the proof of \cite[Proposition 3.2.4]{MaherSisto}, which is stated for a more general definition of matching which allows orientation reversing matchings, but in fact the proof shows this by showing that the probability of an oriented match tends to one.  We give the details for the convenience of the reader.

We shall start by showing that for any $0 < \veps < 1 $, the probability that $\gamma_n$ has a match of size $(1 - \veps) \norm{\gamma_n}$ with $\gamma_\omega$ tends to one as $n \to \infty$.

\begin{prop}\cite[Proposition 3.2.3]{MaherSisto}\label{prop:n w match} Given non-negative constants $\delta$ and $Q$, there is a constant $\kappa > 0$ such that if $G$ is a group acting by isometries on a $\delta$-hyperbolic space $X$, and $\mu$ is a nonelementary probability measure on $G$ with bounded support in $X$, then for almost every sample path $\omega \in (G , \mu)^\Z$ of the random walk, and all $0 < \veps < 1$, the probability that $\gamma_n$ has a subsegment of length $(1 - \veps) \norm{\gamma_n}$ contained in a $\kappa$--neighbourhood of $\gamma_\omega$ tends to one as $n\to\infty$.
\end{prop}

\begin{proof}
Let $\kappa_1$ be a Morse constant for $Q$--quasigeodesics in $X$, i.e. any geodesic connecting points of a $Q$-quasigeodesic is contained in a $\kappa_1$--neighbourhood of the $Q$-quasigeodesic.

As $\mu$ has bounded support in $X$, there are constants $\ell > 0$, $K \ge 0$, and $c < 1$ such that the estimates for positive drift \eqref{eq:drift} and distance from $\gamma_\omega$ \eqref{eq:gp} imply that the probability that both
\begin{equation}\label{eq:dist to g_w} d_X(x_0, \gamma_\omega) \le \veps_1 \norm{\gamma_n} \text{ and } d_X(w_n x_0, \gamma_\omega) \le \veps_1 \norm{\gamma_n}
\end{equation}
is at least $1 - 2 K c^{\veps_1 \ell n } + Kc^n$.  Thus the probability that \eqref{eq:dist to g_w} holds tends to one as $n \to \infty$. If \eqref{eq:dist to g_w} holds, then by thin triangles, $\gamma_n$ has a subgeodesic of length at least $(1 - 2 \veps_1) \norm{\gamma_n}$ contained in a $(\kappa_1 + 2 \delta)$--neighbourhood of $\gamma_\omega$.
%
%
Therefore the result holds with $\veps = 2 \veps_1 $ and $\kappa = \kappa_1 + 2 \delta$.
\end{proof}

We may now complete the proof of Proposition \ref{prop:MS}.

\begin{proof}[Proof of Proposition \ref{prop:MS}] Let $\gamma_n$ be a geodesic from $x_0$ to $w_n x_0$, let $\alpha_{w_n}$ be a $Q$--quasiaxis for $w_n$, and let $\gamma_\omega$ be a bi-infinite $Q$--quasigeodesic determined by the forward and backward limit points of a bi-infinite random walk $\omega$ generated by $\mu$.  Combining Propositions \ref{prop:dh} and \ref{prop:n w match}, there is a constant $\kappa \ge 0$ such that for any $0 < \veps < 1$ the probability that $\gamma_n$ has a subsegment of length $(1 - \veps)\norm{\gamma_n}$ which is contained in a $\kappa$--neighbourhood of $\alpha_{w_n}$, and also in a $\kappa$--neighbourhood of $\gamma_\omega$, tends to one as $n\to\infty$.

Let $g$ be a loxodromic element which lies in the semi-group generated by the support of $\mu$, which we shall denote $\langle \text{supp}(\mu) \rangle_+$, and let $\alpha_g$ be a $Q$-quasiaxis for $g$.  We now show that the probability that a subsegment of $\gamma_\omega$ has a large match with $\alpha_g$ tends to one as the length of the subsegment tends to infinity.

We shall write $\nu$ for the harmonic measure on $\partial X$, and $\rnu$ for the reflected harmonic measure, i.e the harmonic measure arising from the random walk generated by the probability distribution $\rmu(g) = \mu(g^{-1})$.  By assumption, the group element $g$ lies in $\langle \text{supp}(\mu) \rangle_+$, and so the group element $g^{-1}$ lies in $\langle \text{supp}(\rmu) \rangle_+$.  By \cite[Proposition 5.4]{MaherTiozzo}, there is a constant $R_0$ such that $\nu( S_{x_0}( g x_0, R_0 ) ) > 0$ for all $g \in \langle \text{supp}(\mu) \rangle_+$, and so also $\rnu( S_{x_0}( g x_0, R_0 ) ) > 0$ for all $g \in \langle \text{supp}(\rmu) \rangle_+$.

Given $\delta$ and $Q$, there is a constant $\kappa_0 \ge 0$ so that for any constants $\kappa \ge \kappa_0$ and $L \ge 0$, there is an $m$ sufficiently large such that any oriented $Q$-quasigeodesic $\gamma$ from $S_{x_0}(g^{-m} x_0, R_0)$ to $S_{x_0}(g^{m} x_0, R_0)$ has a subsegment of length $L$ which $\kappa$-fellow travels with the $Q$-quasiaxis $\alpha_g$, and whose orientation agrees with the orientation of $\alpha_g$.  In particular, any such $\gamma$ has an $(L, \kappa)$-match with $\alpha_g$.

As $\nu(S_{x_0}(g^m x_0, R_0)) > 0$ and $\rnu(S_{x_0}(g^{-m} x_0, R_0)) > 0$, there is a positive probability $p$ so that $\gamma_\omega$ has an oriented $(L, \kappa)$-match with $\alpha_g$.  Ergodicity now implies that for any $\veps > 0$ the proportion of times $\veps n \le k \le (1 - \veps) n$ for which $\gamma_\omega$ has a subsegment of length at least $L$ which lies in a $K$-neighbourhood of $w_k \gamma_g$ tends to $p$ as $n\to\infty$, for almost all sample paths $\omega$.

The final step is to relate the locations $w_k x_0$ of the random walk to the geodesic $\gamma_n$, by showing that all of the $w_k x_0$ for $\veps n \le k \le (1 - \veps) n$ have nearest point projections to $\gamma_n$ that are far from the endpoints, with asymptotic probability one:

\begin{claim} For any $\veps > 0$, the probability that all $w_k x_0$ with $4 \veps n \le k \le (1 - 4 \veps ) n $ have nearest point projections to $\gamma_n$ which lie distance at least $\veps \norm{\gamma_n}$ from each endpoint tends to one as $n\to\infty$.
\end{claim}

\begin{proof} Let $\ell > 0, K \ge 0$, and $c < 1$ be the constants from the estimates in \eqref{eq:drift}--\eqref{eq:gp}.  Exponential decay for Gromov products \eqref{eq:gp} implies that the probability that $\gp{w_k x_0}{x_0}{w_n x_0} \le \lambda \log n$ for all $1 \le k \le n$ is at least $1 - n K c^{\lambda \log n}$.  For $\lambda$ sufficiently large, this $\to 1$, as $n\to\infty$.
Positive drift with exponential decay \eqref{eq:drift} implies that the probability that $d_X(x_0, w_k x_0) \ge 2 \veps \norm{\gamma_n}$ and $d_X(w_k x_0, w_n x_0) \ge 2 \veps \norm{\gamma_n}$ for all $4 \veps n \le k \le (1 - 4 \veps ) n $ is at least $1 - 2 n K c^{4 \veps n}$, which $\to 1$ as $n\to\infty$.

If the Gromov product $\gp{w_k}{x_0}{w_n x_0} \le \veps \norm{\gamma_n}$, and the nearest point projection of $w_k x_0$ to $\gamma_n$ is within distance $\veps \norm{\gamma_n}$ of $x_0$, then $d_X(x_0, w_k x_0) \le 2 \veps \norm{\gamma_n} + O(\delta)$, and similarly for the other endpoint $w_n x_0$.  Therefore, as the probability that the Gromov products $\gp{w_k x_0}{x_0}{w_n x_0} \le \lambda \log n \le \veps \norm{\gamma_n}$ for all $k$ tends to one, this implies that the probability that all of the nearest point projections of $w_k x_0$ to $\gamma_n$, for $4 \veps n \le k \le (1 - 4 \veps) n $, lie distance at least $\veps \norm{\gamma_n}$ from the endpoints, $\to 1$ as $n\to\infty$.
\end{proof}

This completes the proof of Proposition \ref{prop:MS}.
\end{proof}





\section{Outer space and the attracting tree for a fully irreducible}\label{ss:Trees}

\vskip 2pt

We assume throughout that $r\ge 3$ is an integer, that $F_r$ is the rank-$r$ free group, and that $\out$ is its outer automorphism group. Recall that $\vphi\in\out$ is \emph{fully irreducible} if no positive power of $\vphi$ fixes the conjugacy class of a nontrivial proper free factor of $F_r$. It is proved in \cite{bh92} that fully irreducible outer automorphisms have \emph{train track representatives}.

\vskip10pt

\subsection{Outer space $\cv$}\label{ss:os}

Culler--Vogtmann Outer space was first defined in \cite{cv86}. We refer the reader to \cite{FrancavigliaMartino,b15,v15} for background on Outer space and give only abbreviated discussion here. For $r\ge 2$ we denote the (volume-one normalized) Outer space for $F_r$ by $\cv$.  We think of points of $\cv$ in three different ways. First, points of $\cv$ are equivalence classes of volume-one marked metric graphs $h:R_r\to\G$ where $R_r$ is the $r$-rose, where $\G$ is a finite volume-one metric graph with betti number $b_1(\Gamma)=r$ and with all vertices of degree $\ge 3$, and where $h$ is a homotopy equivalence called a \emph{marking}. Second, a point of $\cv$ can be viewed as a minimal free and discrete isometric action of $F_r$ on an $\mathbb R$-tree $T$ with the quotient metric graph of volume one (where the tree $T$, in the previous picture, is obtained as the universal cover of the metric graph $\G$). In this paper, by an \emph{$F_r$-tree} we will mean an $\mathbb R$-tree $T$ endowed with a minimal nontrivial isometric action of $F_r$, or sometimes the projective class of such an action (when talking about points of compactified Outer space).
Third, a point $T\in \cv$ can be viewed as a length function $||.||_T:F_r\to \mathbb R_{\ge 0}$ where for $w\in F_r$ the number $||w||_T=\inf_{x\in T}d(x,wx)$ is the \emph{translation length} of $w$ in $T$. Each of these three descriptions, in the appropriate sense, uniquely defines a point of $\cv$. One can also think of $\cv$ as the union of the open simplices obtained by varying the lengths of edges of $\G$ for a given $h:R_r\to\G$. By abuse of notation, we will usually denote by $\G$ a point of $\cv$ given by $h:R_r\to\G$ and suppress the mention of the marking $h$.

Outer space $\cv$ admits a natural compactification $\uos=\cv\cup\partial \cv$. Points of $\uos$ are projective (homothety) classes of minimal \emph{very small} $F_r$-actions on $\mathbb R$-trees, \cite{bf94},\cite{cl95}. See \cite{Pa89} for descriptions of several equivalent topologies on $\uos$. The space $\cv$ naturally embeds in $\uos$ as an open dense $\out$-invariant subset, where a point of $\cv$ is identified with its projective class.

We consider $\cv$ with the asymmetric Lipschitz metric (see, for example, \cite{FrancavigliaMartino} for background information). We briefly recall the definition here for completeness. If $h_1:R_r\to\G_1$ and $h_2:R_r\to\G_2$ are two points of $\cv$, a continuous map $f:\Gamma_1\to\Gamma_2$ is called a \emph{difference of markings}  from $\G_1$ to $\G_2$ if $f$ is freely homotopic to $f_2\circ f_1^{-1}$.  The \emph{Lipschitz distance} $\dL(\G_1,\G_2)$ is defined as
\[
\dL(\G_1,\G_2):=\inf_f \log \mathrm{Lip}(f)
\]
where $\mathrm{Lip}(f)$ is the Lipschitz constant of $f$ and where the infimum is taken over differences in markings $f$ from $\G_1$ to $\G_2$.
It is known that $\dL$ is an asymmetric metric, in the sense of \S \ref{ss:msnotation}.

By a \emph{geodesic} in $\cv$ we always mean a directed geodesic with respect to $\dL$, i.e. a path $\gamma:J\to \cv$ (where $J\subseteq \mathbb R$ is a subinterval) such that for any $s<t$ in $J$ we have $\dL(\gamma(s),\gamma(t))=t-s$. It is known that $(\cv,\dL)$ is a directed geodesic space, i.e. for any $\G_1,\G_2\in\cv$, there exists a geodesic (in the above sense) from $\G_1$ to $\G_2$ in $\cv$.

It is sometimes convenient to symmetrize the metric $\dL$
to obtain a (nongeodesic) metric $\ds(x,y)=(\dL(x,y)+\dL(y,x))$. We remark that both metrics $\dL$ and $\ds$ define the standard topology on $\cv$ \cite{FrancavigliaMartino}, and, in particular, $(\cv,\ds)$ is locally compact.




\subsection{Train track maps}

We adopt the same conventions regarding train track maps and train track representatives of elements of $\out$ as in \cite{stablestrata}.
In particular, for a graph $\Gamma$ (which at this point we do not have to assume to be finite) and a vertex $x$ of $\Gamma$, the \emph{directions} at $x$ in $\Gamma$ are germs of initial segments of edges emanating from $x$. We refer the reader to~\cite{bh92,bogop08,bers,dkl15} for more detailed background on train track maps in the free group outer automorphism context.

\subsection{Attracting and repelling trees $T_+^{\vphi}$, $T_-^{\vphi}$ }\label{subs:Trees}

Let $\vphi \in \out$ be a fully irreducible outer automorphism. Then $\vphi$ acts on $\uos$ with ``North-South dynamics'' (see \cite{ll03}). As elements of $\partial\cv$, the attractor and repeller for the $\out$-action are $F_r$-trees, we denote them respectively by $T_+^{\vphi}$ and $T_-^{\vphi}$.

We recall from \cite{gjll} a concrete construction of $T_+^{\varphi}$.
Let $g\colon \Gamma \to \Gamma$ be a train track representative of $\varphi$ and $\tilde{\Gamma}$ the universal cover of $\Gamma$ equipped with a distance function $\tilde{d}$ lifted from $\Gamma$. The fundamental group, $F_r$, acts by deck transformations, hence isometries, on $\tilde{\Gamma}$. A lift $\tilde{g}$ of $g$ is associated to the unique automorphism $\Phi$ representing $\varphi$ satisfying that, for each $w \in F_r$ and $x \in \tilde{\Gamma}$, we have $\Phi(w)\tilde{g}(x)=\tilde{g}(wx)$. Define the pseudo-distance $d_{\infty}$ on $\tilde{\Gamma}$ by $\lim_{k \to +\infty} d_k$, where

$$d_k(x,y)=\frac{d(\tilde{g}^k(x),\tilde{g}^k(y))}{\lambda^k}$$

\noindent for each $x,y \in \tilde{\Gamma}$. Then $T_+$ is the $F_r$-tree defined by identifying each pairs of points $x,y \in \tilde{\Gamma}$ such that $d_{\infty}(x,y)=0$. It is known from \cite{bf94}, for example, that $T_+^{\vphi}$ is the $F_r$-equivariant Gromov limit of the trees $\widetilde{\Gamma_k}$ obtained from $\tilde{\Gamma}$ by identifying distance-0 points in $\tilde{\Gamma}$ endowed with the respective pseudo-distance $d_k$.

\begin{rk}[Left and right actions on $\cv$ and the direction of the folding lines]\label{r:RL}
The standard action of $\out$ on $\cv$ and on $\uos$ is a right action. At the level of length functions, this action is described as follows. For a tree $T$, an outer automorphism $\vphi\in\out$, and an element $w\in F_r$ we have $||w||_{T\vphi}:=||\vphi(w)||_T$. In the language of actions on trees, for an $F_r$-tree $T$ and an outer automorphism  $\vphi\in\out$, the $F_r$-tree $T\vphi$ is defined as the tree $T$ with the action of $F_r$ twisted via $\vphi$: For $x\in T$ and $w\in F_N$ we have $w\underset{T\vphi}{\cdot} x:=\vphi(w)\underset{T}{\cdot} x$.

If $\vphi\in \out$ is fully irreducible and $T_+^{\varphi}\in \uos$ is the attracting tree of $\vphi$, then, projectively, we have $\lim_{n\to\infty} T_0\vphi^n=T_+^{\varphi}$ in  $\uos$, for any $T_0\in \cv$.  In particular, if $\gamma$ is a periodic folding axis for $\vphi$ in $\cv$ (see \S \ref{ss:foldlines} below) with $\lim_{t\to\infty} \gamma(t)= T_+^{\varphi}$ (see \S \ref{ss:foldlines} below) then for every $S\in \gamma$ we have $\lim_{n\to\infty} S\vphi^n=T_+^{\varphi}$.

There is also a standard way to convert this right action of $\out$ on $\cv$ and on $\uos$ to a left action. Namely, for a tree $T$ and an element $\vphi\in\out$ set $\vphi\cdot T:=T\vphi^{-1}$.  We need to work with both these right and left actions.  Note that if $\vphi\in \out$ is fully irreducible with a periodic folding axis $\gamma$ in $\cv$ (see the discussion of axes below) then  for any $S\in \gamma$ we have
\[
\lim_{t\to\infty} \gamma(t)=\lim_{n\to\infty}\vphi^{-n}S=\lim_{n\to\infty}S\vphi^n=T_+^{\varphi}
\]
in $\uos$.
\end{rk}

\subsection{Ageometric fully irreducible outer automorphisms and PNPs}\label{ss:ageo}

\vskip2pt

As in \cite[Definition 0.1]{LevittPaulin}, we call an $\R$-tree \emph{geometric} if it arises as the dual to a measured foliation on some finite simplicial 2-complex. A fully irreducible $\vphi\in\out$ is \emph{ageometric} if $T_+^{\vphi}$ is not geometric. If $T_+^{\vphi}$ is geometric and $\vphi$ is induced by a surface homeomorphism, then $\vphi$ is called \emph{geometric}. One otherwise calls a fully irreducible outer automorphism $\vphi$ \emph{parageometric}, i.e. when $T_+^{\vphi}$ is geometric, but $\vphi$ is not induced by a surface homeomorphism.

While the definition of a \emph{(periodic) Nielsen path}, or \emph{PNP}, will not directly be of use to us, it will be relevant that each rotationless power of each ageometric fully irreducible outer automorphism has a PNP-free train track representative. In particular, by \cite[Theorem 3.2]{bf94}, for a fully irreducible $\vphi\in\out$, we have that $T_+^{\vphi}$ is geometric if and only if the \cite{bh92} ``stable'' train track representative of $\vphi$ contains a PNP. Using the ``stable'' PNP-free train track representative for an ageometric fully irreducible outer automorphism allows us, for example, to use the simplified definition of a principal vertex stated in \S \ref{ss:principal} and the simplified ideal Whitehead graph definition of \S \ref{ss:iwgrelationships}.

\vskip10pt

\section{Laminations and ideal Whitehead graphs}\label{s:IWGs}

\vskip2pt

\begin{conv}\label{conv:rot}
We continue to assume throughout this section that $r\ge 3$ is an integer, that $F_r$ is the rank-$r$ free group, and that $\out$ is its outer automorphism group. We add an assumption that $\vphi\in\out$ is an ageometric fully irreducible outer automorphism and that $g\from\G\to\G$ is a PNP-free irreducible train track representative of $\vphi$. We use $\tilde g\from \tilde\G \to \tilde\G$ to denote a lift of $g$ to the universal cover. Finally, for simplicity, we assume that $\vphi$, hence also all of its train track representatives, are \emph{rotationless}.

The theory of rotationless (outer) automorphisms can be found in \cite{fh11} or \cite{hm11}. We will use the facts that each positive power of this rotationless $g$ will itself be rotationless, and that anything that is fixed (in the sense of \S \ref{ss:principal}) by a positive power $g^k$ will in fact be fixed by $g$ itself. Thus, for $g$ there is no difference between ``periodic'' and ``fixed'' for vertices, directions, etc.
\end{conv}

We remark that the PNP-free and rotationless assumptions for $\vphi$ and $g$ are not restrictive since every fully irreducible outer automorphism has a positive rotationless power, and since (see \S \ref{ss:ageo}) every ageometric fully irreducible outer automorphism admits a PNP-free train track representative.

\vskip10pt

\subsection{Laminations}\label{ss:laminations}


\vskip10pt

The attracting lamination, $\Lambda_{\vphi}$, for a fully irreducible $\vphi\in\out$ is defined in \cite{bfh97}. The lamination records the limiting behavior of a cyclic word under repeated application of an automorphism, and hence has intimate connections to tiling and substitution theory, for example (see \cite{abhs06}, \cite{chl08I}). We will be most interested in its ``singularity structure,'' as described in \S \ref{ss:iwgdfs} and \S \ref{ss:iwgrelationships}.


Before proceeding further, we warn the reader of a bit of notational discrepancy in the literature. In most places $\Lambda_{\vphi}$ comes endowed with a ``$+$'' sign, but in \cite{hm11} $\Lambda_{\vphi}$ comes instead endowed with a ``$-$'' sign. Since we only use the attracting lamination, we leave the sign out altogether.

The attracting lamination can equally be realized as certain unordered pairs (\emph{leaves}) of distinct boundary points in $F_r$, in $T_+^{\vphi}$, or in the boundary of any point of $\cv$ whose quotient under the $F_r$ action carries a train track representative for $\vphi$ (points that are ``train tracks'' in the language of \cite{hm11} or \cite{loneaxes}).

For a definition of $\Lambda_{\vphi}$ in terms of a train track representative $g\from\G\to\G$, and most closely relevant to our context, we refer the reader to \cite[pg. 36]{hm11}. We give a brief description here.
The \emph{realization} $\Lambda(g)$ of $\Lambda_{\vphi}$ in $\G$ is the collection of all lines (\emph{leaves}) realized as bi-infinite immersed edge-paths $\alpha$ in $\G$
\[
\alpha=\dots,  e_{-n},\dots, e_{-1}, e_0, e_1,\dots e_n, \dots
\]
(where $e_i$ are edges of $\G$) with the property that for every finite subpath $\beta$ of $\alpha$ there exists an integer $k\ge 1$ and an edge $e$ of $\G$ such that $\beta$ is a subpath of $g^k(e)$.

The lifts of leaves of $\Lambda(g)$ to $\tilde \G$ are also referred to as lamination leaves. Every such lift connects a pair of distinct points in $\partial\tilde{\G}$ and the set of the resulting pairs forms a closed flip-invariant $F_r$-invariant subset of $\partial^2 \tilde{\G}=\partial\tilde{\G}\times \partial\tilde{\G} - \Delta$ which gives another way of viewing $\Lambda(g)$.  Via a natural identification $\partial \tilde{\G}\cong \partial F_r$ and, correspondingly, $\partial^2 \tilde{\G}\cong \partial^2 F_r$, this allows one to associate to $\Lambda(g)$ a subset  $\Lambda_{\vphi}\subseteq \partial^2 F_r$ that turns out to be independent of the choice of $g$ and to depend on $\vphi$ only.

On pg. 29-30 of \cite{hm11}, Handel and Mosher define a ``direct limit" map $g_\infty\from \tilde{\Gamma}\to T_+^{\vphi}$ (there called $f_g$) that is an isometry when restricted to each edge of $\tilde{\Gamma}$ (when $\tilde \G$ is equipped with the natural ``eigenmetric" determined by the transition matrix of $g$). The map $g_\infty$ is well-defined up to isometric conjugacy of the $F_r$ action on $T_+^{\vphi}$. By \cite[Lemma 2.21]{hm11} we know that the restriction of $g_\infty$ to each leaf is an isometry and that the collection in $T_+^{\vphi}$ of $g_\infty$-images of leaves is independent of the choice of train track representative $g$ of $\vphi$. This collection of leaf images will be the \emph{realization of $\Lambda_{\vphi}$ in $T_+^{\vphi}$}.


\vskip10pt

\subsection{Principal vertices and lifts}\label{ss:principal}

The notions of principal points and principal lifts were first introduced in \cite{fh11} to mimic the corresponding surface theory notions. We first establish certain relevant terminology.

Since both $T_+^{\vphi}$ and $\tilde \G$ are $\R$-trees, we freely use $\R$-tree terminology. For a general $\R$-tree $T$ we have the following. A point $p\in T$ is called a \emph{branch point} if $T\backslash p$ has $\ge 3$ components. Each of these components is called a \emph{direction} at $p$.  A \emph{turn} at $p$ is an unordered pair of directions $\{d_i, d_j\}$ at $p$.

Recall that $\vphi$ and $g$ are as in Convention \ref{conv:rot}.
A \emph{$g$-fixed point} (which, since we are in the absence of PNPs, is necessarily a vertex) in $\G$ is \emph{principal} if it has $\ge 3$ $g$-fixed directions. A point $\tilde{v}\in\tilde{\G}$ is called \emph{principal} if it is a lift of a principal point or, equivalently, if some  positive power of the lift $\tilde{g}$ fixes $\tilde{v}$ and $\ge 3$ directions at $\tilde{v}$. (See, for example, \cite[pg. 27-28]{hm11}).
A lift $\tilde{g} \colon \tilde{\G} \to \tilde{\G}$ is \emph{principal} when its boundary extension, denoted $\hat{g}$, has $\ge 3$ nonrepelling fixed points.


\vskip10pt

\subsection{Local/stable Whitehead graphs}\label{ss:iwgdfs}

In \cite{hm11}, Handel and Mosher define the ideal Whitehead graph for an arbitrary nongeometric fully irreducible $\vphi\in\out$. Here we will restrict our consideration to the case of ageometric fully irreducible outer automorphism. since that is the only case relevant to the present paper.  We recall that $\vphi$ and $g$ are as in Convention~\ref{conv:rot}.


We define the following graphs:

\begin{itemize}
\item For a vertex $v$ of $\G$, the \emph{local Whitehead graph} $LW(g,v)$ has a vertex for each direction at $v$ and an edge connecting the vertices corresponding to the pair of directions $\{d_1,d_2\}$ if the turn $\{d_1,d_2\}$ is taken by an image under some $g^k$ of some, hence any, edge of $\G$ (equivalently is taken by a leaf of $\Lambda(g)$).

\item The \emph{stable Whitehead graph} $SW(g,x)$ at a principal vertex
$x$ is then the induced subgraph of $LW(g,x)$ obtained by restricting to the  $g$-fixed directions.
\end{itemize}

For a principal vertex $\tilde x\in\tilde \G$, one similarly defines $\widetilde{SW}(\tilde x)$ to have a vertex for each direction at $\tilde x$ fixed by the principal lift $\tilde g \from \tilde\G\to\tilde\G$ fixing $\tilde x$ and an edge connecting the corresponding pair of vertices for each turn taken by a lamination leaf.  See \S 3 of \cite{hm11} for the detailed definition of $\widetilde{SW}(\tilde x)$.
  Note, however, that the graph $\widetilde{SW}(\tilde x)$ is naturally isomorphic to $SW(g,x)$, where $\tilde x$ is a lift of $x$.

\subsection{Ideal Whitehead graphs}\label{ss:iwgrelationships}

In \S 3 of \cite{hm11} Handel and Mosher define the ideal Whitehead graph $IW(\vphi)$ and give several equivalent descriptions of it. We use their description in the train track language (under our assumption of no PNPs) as the main definition.  See \cite[Lemma 3.1, Corollary 3.2]{hm11} for additional details. Note that if $\tilde g$ fixes two distinct vertices in $\tilde \G$ then the geodesic between them projects to a PNP in $\G$ \cite[Lemma 3.1]{hm11}, contradicting our no PNP assumption on $g$. Thus the no PNPs assumption on $g$ means that each principal lift $\tilde g$ of $g$ fixes a unique principal vertex $\tilde x$ in $\tilde \G$, which we call the \emph{center} of $\tilde g$. We will also sometimes say that $\tilde g$ is \emph{centered} at $\tilde x$.  In this case, we define\footnote{\cite{hm11} give a different definition of $W(\tilde g)$ which in our setting is equivalent to the one given here.}
\[
W(\tilde g):=\widetilde{SW}(\tilde x).
\]

Recall that two principal lifts $\tilde g$, $\tilde g'$ of $g$ are considered equivalent, $\tilde g\sim\tilde g'$, if there exists some $w\in F_r$ such that $\tilde g'=w\circ \tilde g\circ w^{-1}$ (for the left translation covering action of $F_r$ on $\tilde \G$). In this case (again, under the no PNPs assumption), there is a bijective correspondence between equivalence classes of principal lifts of $g$ and the set of principal vertices of $\G$, see \S 2 of \cite{hm11} for the details.

We now state the definition of the ideal Whitehead graph of $\vphi$, following the discussion on p. 42 in \S 3.1 of \cite{hm11}.

\begin{pd}[Ideal Whitehead graph]
We define the \emph{ideal Whitehead graph} as follows.
\begin{enumerate}

\item
Let $\vphi$ and $g$ be as in Convention \ref{conv:rot}.

We define the \emph{ideal Whitehead graph} $IW(\vphi)$ of $\vphi$ as the disjoint union of the $W(\tilde g)$ taken over the equivalence classes of all principal lifts $\tilde g$ of $g$.

In view of the discussion above and using \cite[Lemma 3.1, Corollary 3.2]{hm11}, under our assumption of no PNPs, we have
\[
IW(\vphi)=\sqcup  SW(g,x)
\]
where the union is taken over all principal vertices $x\in \G$.

\item Handel and Mosher (see also \cite{Thesis}) prove that this definition of $IW(\vphi)$ does not depend on the choice of $g$, is an invariant of the conjugacy class of $\vphi$ in $\out$, and is preserved under taking positive powers. 
    Recall that each ageometric fully irreducible $\vphi\in\out$ admits an expanding irreducible train track representative with no PNPs, and therefore the above definition applies to all rotationless ageometric fully irreducible outer automorphisms.

\item Let $\vphi\in \out$ be an arbitrary ageometric fully irreducible outer automorphism. Then there exists some integer $k\ge 1$ such that $\vphi^k$ is rotationless. Set $IW(\vphi):=IW(\vphi^k)$.
\end{enumerate}
\end{pd}

\subsection{Ideal Whitehead graph and branch points in the attracting tree}

In \cite{hm11}, for a nongeometric fully irreducible $\vphi\in\out$,  Handel and Mosher also give an equivalent description of $IW(\vphi)$ in terms of directions at branch points in  $T_+^{\vphi}$ and the turns ``taken" by $\Lambda_{\vphi}$ there. We will not need the full strength of this description here and will only state the simplified versions of their results required for our purposes. Recall that $g_{\infty}$ is the map $f_g$ of \cite[pg. 29-30]{hm11}.

\begin{prop}\label{p:br}
Let $\vphi\in\out$ be a rotationless ageometric fully irreducible outer automorphism represented by an expanding irreducible train track map $g:\G\to\G$ with no PNPs.
Let $g_{\infty}:\tilde\G\to  T_+^{\vphi}$ be the $F_r$-equivariant ``direct limit'' map.
Then the following hold:

\begin{enumerate}
\item Let $\tilde g$ be a principal lift of $g$ centered at a principal vertex $\tilde x$ of $\tilde \G$. Then $b=g_{\infty}(\tilde x)$ is a branch point of $T_+^{\vphi}$ and the map $g_{\infty}$ induces a bijection between the set of $\tilde g$-fixed directions at $\tilde x$ in $\tilde \G$ (which are precisely the vertices of $\widetilde{SW}(\tilde x)$) and the set of directions (i.e. connected components of $T_+^{\vphi}-\{b\}$) of $T_+^{\vphi}$ at $b$.
\item Every branch point $b\in T_+^{\vphi}$ arises as in part (1).
\end{enumerate}

\end{prop}
\begin{proof}
Part (1) follows from \cite[Lemma~3.4]{hm11}. Part (2) follows from \cite[Lemma~2.16(1)]{hm11}.
\end{proof}

\subsection{The rotationless index and an alternate ageometric characterization}\label{ss:indices}

The notion of an index $\ind{(\varphi)}$ for an outer automorphism $\varphi \in Out(F_r)$ was first introduced in \cite{gjll}. This notion is not in general invariant under taking powers. \cite{hm11} introduces the notion of a rotationless index (there just called the index sum) $i(\varphi)$ for a fully irreducible $\varphi \in \Out(F_r)$. Let $\varphi$ be a nongeometric fully irreducible outer automorphism.
For each component $C_i$ of $IW(\vphi)$, let $k_i$ denote the number of vertices of $C_i$. Then the \emph{rotationless index} is defined as
$i(\vphi) := \sum 1-\frac{k_i}{2}$. It follows from~\cite[Lemma 3.4]{hm11} that for a rotationless nongeometric fully irreducible $\varphi \in \Out(F_r)$, the two notions differ only by a change of sign.

The rotationless index can be used to determine whether a fully irreducible outer automorphism is ageometric. In particular, a fully irreducible $\vphi\in\out$ is ageometric if any only if $0>i(\vphi)>1-r$. While the characterization follows fairly easily from a sequence of previous results, it is finally more or less stated in \cite{kp15} as Corollary 2.36. A justification and explanation of the fact can be found sprinkled throughout \cite{kp15} or in \cite[\S 2.9]{loneaxes}.

\vskip10pt

\section{Fold lines, axis bundles, and geodesics in $\cv$}\label{s:geodesics}
The proof of our main theorem requires us to understand the behavior of an axis of a random outer automorphism. As such, we need a detailed account of these axes.

\subsection{Fold lines and axis bundles}\label{ss:foldlines}


Let $\widehat\cv$ denote the ``unprojectivized'' version of $\cv$ and $p: \widehat\cv\to\cv$ the projection map. Then a
A \emph{fold line} in $\cv$ is the image under $p$ of a continuous, injective, proper function from $\R$ to $\hat\cv$ that is defined by a continuous 1-parameter family of marked graphs $t \to \Gamma_t$ and a family of homotopy equivalences $h_{ts} \colon \Gamma_s \to \Gamma_t$ defined for $s \leq t \in \R$, each marking-preserving, and satisfying:
~\\
\vspace{-\baselineskip}
\begin{itemize}
\item $h_{ts}$ is a local isometry on each edge for all $s \leq t \in \mathbb{R}$ and 
\item $h_{ut} \circ h_{ts} = h_{us}$ for all $s \leq t \leq u \in \mathbb{R}$ and $h_{ss} \colon \Gamma_s \to \Gamma_s$ is the identity for all $s \in \R$.
\end{itemize}

A fold line in Outer space $\mathbb{R} \to \os$ is said to be \emph{simple} if there exists a subdivision of $\mathbb R$ by points $(t_i)_{i\in \mathbb Z}$
\[
  \dots t_{i-1}< t_i <t_{i+1} \dots
\]
such that $\lim_{i\to\infty} {t_i}=\infty$, $\lim_{i\to-\infty} t_i=-\infty$, and the following holds:

For each $i\in \mathbb Z$ there exist distinct edges $e,e'$ in $\Gamma_{t_i}$, with a common initial vertex, such that: For each $s\in (t_i,t_{i+1}]$ the map $h_{s t_i}\colon \Gamma_{t_i}\to \Gamma_s$ identifies an initial segment of $e$ with an initial segment of $e'$, with no other identifications (that is, $h_{s t_i}$ is injective on the complement of those two initial segments in $\Gamma_{t_i}$).

\begin{rk}\label{r:sg} By \cite[Lemma 2.27]{stablestrata}, simple periodic fold lines in $\cv$, defined by  expanding irreducible train track representatives of outer automorphisms $\vphi\in\out$, are geodesics in $\cv$.
\end{rk}

In \cite{hm11}, Handel and Mosher define the axis bundle of a fully irreducible, in analogy with the Teichm\"uller axis for a pseudo-Anosov:


Suppose that $\vphi\in\out$ is a nongeometric fully irreducible outer automorphism.
Then the \emph{axis bundle} $\mathcal{A}_{\varphi}$ is the union of the images of all fold lines $\mL \colon \mathbb{R} \to \cv$ such that $\mL$(t) converges in $\uos$ to $T_{-}^{\varphi}$ as $t \to -\infty$ and to $T_{+}^{\varphi}$ as $t \to +\infty$.


We call the fold lines of the axis bundle \emph{axes}. An axis $\mL$ in $\cv$ for $\vphi$ is \emph{periodic} if there exists an $R>0$ such that for every $t\in \mathbb R$ we have $\mL(t+R)=\mL(t)\vphi$. Thus in this case
$\mL$ is $\vphi$-invariant and for every $x_0\in\gamma$ we have $\lim_{n\to\infty} x_0\vphi^n =T_{+}^{\varphi}$ and $\lim_{n\to\infty} x_0\vphi^{-n}=T_{-}^{\varphi}$ in $\uos$. (We recall the conventions regarding right and left actions of $\out$ explained in Remark~\ref{r:RL}). In more generally understanding ``periodic fold lines,''  of particular relevance here are the descriptions given in  \cite[Definitions 2.18-2.20]{stablestrata}, though these notions date back to Stallings \cite{s83} and Skora \cite{s89}.

In the circumstance where $\mathcal{A}_{\vphi}$ is formed from only a single fold line, we call this fold line a \emph{lone axis} and we say that $\vphi$ is a \emph{lone axis} fully irreducible outer automorphism. It is noteworthy that the lone axis will always be a periodic fold line for $\vphi$.

\begin{rk}[Fully irreducible outer automorphisms have simple axes]\label{r:simpleaxes} While not every periodic axis of a fully irreducible outer automorphism is simple, the axes constructed from Stallings fold decompositions, as explained in \cite[Definitions 2.18-2.20]{stablestrata} for example, are simple. In particular, every fully irreducible outer automorphism has at least one simple axis.
\end{rk}

Mosher--Pfaff \cite{loneaxes} gives a necessary and sufficient condition for an ageometric fully irreducible outer automorphism to have a lone axis.

\begin{thm}[\cite{loneaxes} Theorem 4.7] \label{lem:ue} The axis bundle of an ageometric fully irreducible outer automorphism $\varphi \in \out$ is a unique axis precisely if both of the following two conditions hold:
~\\
\vspace{-6mm}
\begin{enumerate}
\item the rotationless index satisfies $i(\varphi) = \frac{3}{2}-r$ and 
\item no component of the ideal Whitehead graph $IW(\varphi)$ has a cut vertex.
\end{enumerate}
\end{thm}

One may note that, using the definition of \S \ref{ss:indices}, the first condition can be replaced by an ideal Whitehead graph condition.

\begin{rk}[Lone axes are simple greedy fold lines]\label{r:greedyfoldlines} It will be relevant throughout this paper that the lone axis of a lone axis fully irreducible outer automorphism is always a ``greedy fold line'' (also called a ``fast folding path" in \cite[Definition~5.7]{FrancavigliaMartino}).
One can find the full definition of a greedy fold path in \S $2$ of \cite{bf11} and on pg. 33 of \cite{DahmaniHorbez}. We omit the precise definition here but the idea is that at every time $t$ of the folding process one folds all the turns that can be folded then, at uniform speed. The lone axis of a lone axis fully irreducible outer automorphism is greedy because at every time $t$ there is a unique turn to fold. By \cite[Theorem~5.6]{FrancavigliaMartino}, greedy fold lines are geodesics in $\cv$, and hence lone axes of lone axis fully irreducible outer automorphisms are geodesics as well. This fact provides another explanation of these axes being geodesic, alternative to the use of \cite[Lemma 2.27]{stablestrata} mentioned in Remark~\ref{r:sg} above.
 The lone axis of a lone axis fully irreducible outer automorphism is also always a simple fold line.
 \end{rk}

The following lemma states that any geodesic which fellow travels an axis has the same forward and backward limits. It will be important later.

\begin{lem} \label{lem:closetoaxes}
Suppose that $\gamma$ is any geodesic axis of a nongeometric fully irreducible $\vphi \in \out$ and that a bi-infinite geodesic $\gamma'$ satisfies that, for some $\rho>0$, we have $\ds(\gamma(t), \gamma'(t))\le \rho$ for all $t\in \mathbb R$. Then $\gamma'(t)\to T^{\vphi}_+$ as $t\to +\infty$ and $\gamma'(t)\to T^{\vphi}_-$ as $t\to -\infty$.
\end{lem}

\begin{proof}
Suppose that $\lim_{t \to\infty}\gamma'(t)\ne T^{\vphi}_+$ in $\uos$. Since $\uos$ is compact, there exists some $T'\in \partial \cv$, with $T'\ne  T^{\vphi}_+$, and a sequence $t_i\ge 0$, with $\lim_{i\to\infty} t_i=\infty$, such that $\lim_{i\to\infty}\gamma'(t_i)=T'$ in $\uos$. By the assumptions on $\gamma'$, for each $i\ge 1$, we have that $\ds(\gamma'(t_i), \gamma(t_i))\le \rho$. Since $\lim_{i\to\infty} t_i=\infty$, we have that $\lim_{i\to\infty} \gamma(t_i)=T^{\vphi}_+$. Because of the condition that $\ds(\gamma(t_i), \gamma'(t_i))\le \rho$, it follows that for the (unprojectivized) $\mathbb R$-trees $T^{\vphi}_+$ and $T'$ there exists some $c\ge 1$ such that for every $w\in F_r$
\[
\frac{1}{c}||w||_{T^{\vphi}_+}\le ||w||_{T'}\le c ||w||_{T^{\vphi}_+}.
\]
Hence $L^2(T^{\vphi}_+)=L^2(T')$ where $L^2(.)$ is the ``dual lamination" of a tree in $\uos$, see \cite{kl10} for the precise definition. Therefore, by \cite[Theorem 1.3(2)]{kl10}, we have $T'=T^{\vphi}_+$ in $\uos$. This contradicts the assumption that $T'\ne T^{\vphi}_+$. Thus $\lim_{t\to\infty}\gamma'(t)= T^{\vphi}_+$, as claimed. The argument for $\lim_{t\to -\infty}\gamma'(t)= T^{\vphi}_-$ is similar.
\end{proof}

\subsection{The space of geodesics}

\vskip10pt

\begin{df}[Fellow travelling]
Let $R\ge 0$ and $\rho\ge 0$. Let $\gamma: I\to \cv$ and $\gamma': J\to \cv$ be parameterized geodesics in $\cv$ (where $I,J\subseteq \mathbb R$ are some intervals).

(1) Let $t_0, t_0'\in \mathbb R$ be such that $[t_0,t_0+R]\subseteq I$, and $[t_0', t_0'+R]\subseteq J$, and for each $s\in [0,R]$,
\[
\ds(\gamma(t_0+s),\gamma'(t_0'+s))\le \rho.
\]
We then say that $\gamma|_{[t_0,t_0+R]}$ and $\gamma|_{[t_0',t_0'+R]}$ \emph{$\rho$-fellow travel}.

(2) We say that $\gamma$ and $\gamma'$  \emph{$\rho$-fellow travel for length $R$} if there exist $t_0,t_0'$ such that $\gamma|_{[t_0,t_0+R]}$ and $\gamma|_{[t_0',t_0'+R]}$ $\rho$-fellow travel.

(3) For a point $x=\gamma(t_1)$ on $\gamma$, we say that $\gamma$ and $\gamma'$  \emph{$\rho$-fellow travel for length $R$ at a segment centered at $x$} if there is some $t_0'\in \mathbb R$ such that   $\gamma|_{[t_1-\frac{R}{2},t_1+\frac{R}{2}]}$ and $\gamma'|_{[t_0',t_0'+R]}$ $\rho$-fellow travel.

\end{df}


\begin{df}[Space of geodesics $Fl_r$]\label{d:SpaceOfGeodesics}
We denote by $Fl_r$ the set of all parameterized bi-infinite geodesic fold lines $\mL:\mathbb R\to\cv$  in $\cv$. For $\mL\in Fl_r$, as in \cite{stablestrata}, we let $N(\mL,\varepsilon)$ denote the $\veps$-neighborhood of $\mL$ in $\os$ with respect to the symmetrized Lipschitz metric on $\os$.

\begin{enumerate}
\item[(a)] For $R>0, \veps>0$, and $\mL\in Fl_r$, we denote by $B(\mL,R,\veps) \subset Fl_r$ the set of all $\mL'\in Fl_r$ such that $\mL$ and $\mL'$ $\veps$-fellow travel for length $R$.\footnote{Note that this definition of $B(\mL,R,\veps)$ is slightly different from the definition of $B(\mL,R,\veps)$ given in \cite{stablestrata}, where the issue of directions on $\mL$ and $\mL'$ was ignored due to insufficiently careful phrasing. The definition of $B(\mL,R,\veps)$ is the one which is actually meant in \cite{stablestrata} and is needed in the present paper.}

\item[(b)] We topologize $Fl_r$ by using, for each $\mL\in Fl_r$, the family of sets $\{B(\mL,R,\veps)\}_{\veps>0, R\ge 1}$ as the basis of neighborhoods of $\mL$ in $Fl_r$. We call this topology $\tau$. And we write $\mL_i \to \mL$ to mean that a sequence of fold lines $\{\mL_i\}$ converges to $\mL$ in this topology.

\end{enumerate}

\end{df}

The following lemma is useful in terms of understanding the topology $\tau$ on $Fl_r$.

\begin{lem}\label{l:fell}
Let $\mL\in FL_r$,  where $\mL$ is contained in some thick part of $\cv$, and let $R\ge 1, \veps>0$ be such that $R> 4\veps$.  Let $\mL':J\to\cv$ be a geodesic where $J\subseteq \mathbb R$ is an interval of real numbers. Let $t_1<t_2$ and let $t_1', t_2'\in J$ be such that $t_2'-t_1'=R$, such that $\ds(\mL(t_1), \mL'(t_1'))\le \veps$, such that $\ds(\mL(t_2), \mL'(t_2'))\le \veps$, and such that $\mL'\left([t_1',t_2']\right)\subseteq  N(\mL,\veps)$.

Then $R-2\veps\le t_2-t_1\le R+2\veps$ and there exists a constanct $c>0$, depending only on the thickness of $\gamma$, and such that $\mL'|_{[t_1',t_2']}$ has a subsegment of length $R-c\veps$ which $c\veps$-fellow travels a subsegment of $\mL$ of length $c\veps$.


In particular, if $J=\mathbb R$ and $\mL'\in FL_r$, then $\mL'\in B(\mL,R-c\veps,c\veps)$.
\end{lem}

\begin{proof}
Since $\mL$ is contained in some thick part of $\cv$, there exists a constant $c_0>0$ depending only on the thickness of $\mL$ (see \cite[Theorem 24]{ab12}) such that for any $s_1,s_2\in \mathbb R$ we have
\[\dL(\mL(s_1),\mL(s_2))\le \ds(\mL(s_1),\mL(s_2))\le c_0|s_1-s_2|.\]

Using the triangle inequality for $\dL$ and the fact that $\dL\le \ds$ we see that $R-2\veps\le t_2-t_1\le R+2\veps$.

Moreover, for the same reason, if $s,s'\in [0,R]$ are such that
\[
\ds(\mL'(t_1'+s'),\mL(t_1+s))\le\veps,\tag{$\ddag$}
\]
then $|s-s'|\le 2\veps$. Hence, by the choice of $c_0$, in the case of $(\ddag)$ we have
\[
\ds(\mL'(t_1'+s'),\mL(t_1+s'))\le 2\veps(1+c_0).\tag{$\ddag\ddag$}
\]

Let
$$Z=\{s'\in [0,R]\mid \exists ~s\in [0,R] \text{ with } \ds \mL'(t_1'+s'),\mL(t_1+s))\le\veps\}.$$
Thus $Z$ is a closed subset of $[0,R]$ containing $0$ and $R$. Set $\Omega=[0,R]\setminus Z$. Then $\Omega$ is an open subset of $[0,R]$, not containing $0$ or $R$, so that $\Omega$ must be a disjoint union of open intervals. Suppose that $J=(a',b')$, with $a'<b'$, is one of these intervals. Thus $a',b'\in Z$. Hence, by $(\ddag\ddag)$, $\ds(\mL'(t_1'+a'),\mL(t_1+a'))\le 2\veps(1+c_0)$ and $\ds(\mL'(t_1'+b'),\mL(t_1+b'))\le 2\veps(1+c_0)$.

Recall that, by our assumptions, $\mL'\left([t_1',t_2']\right)\subseteq  N(\mL,\veps)$. Thus, since $J\subseteq [0,R]\setminus Z$, either
\[
\ds (\mL'(t_1'+a'), \mL((-\infty,t_1]))\le\veps \text{ or } \ds (\mL'(t_1'+a'), \mL([t_2,\infty)))\le\veps
\]
 and, similarly, either
\[
\ds (\mL'(t_1'+b'), \mL((-\infty,t_1]))\le\veps \text{ or } \ds (\mL'(t_1'+b'), \mL([t_2,\infty)))\le\veps.
\]

We consider separately 3 cases.

\noindent {\bf Case 1.}  There exist $y_a,y_b<t_1$ such that $\ds(\mL(y_a), \mL'(t_1'+a'))\le \veps$ and $\ds(\mL(y_b), \mL'(t_1'+b'))\le \veps$.

Then $\ds(\mL(y_a), \mL(t_1+a')))\le \veps+ 2\veps(1+c_0)$ and $\ds(\mL(y_b), \mL(t_1+b')))\le \veps+ 2\veps(1+c_0)$.  Hence $|t_1+a'-y_a|\le \veps+ 2\veps(1+c_0)$ and $|t_2+b'-y_b|\le \veps+ 2\veps(1+c_0)$
which implies that $a',b'\le \veps+ 2\veps(1+c_0)$.

\noindent {\bf Case 2.} There exist $y_a, y_b>t_2$ such that $\ds(\mL(y_a), \mL'(t_1'+a'))\le \veps$ and $\ds(\mL(y_b), \mL'(t_1'+b'))\le \veps$.

Then by a similar argument as in Case 1, we see that $|R-a'|, |R-b'|\le \veps+ 2\veps(1+c_0)$.

\noindent {\bf Case 3.} There exist values $y_1\le t_1$ and $y_2\ge t_2$ such that either
$$\ds(\mL(y_1), \mL'(t_1'+a'))\le \veps \text{ and } \ds(\mL(y_2), \mL'(t_1'+b'))\le \veps,$$
or
$$\ds(\mL(y_1), \mL'(t_1'+b'))\le \veps \text{ and } \ds(\mL(y_2), \mL'(t_1'+a'))\le \veps.$$

Let
$$Q_1=\{ q\in [a',b'] \mid \exists ~y\le t_1 \text{ with }  \ds(\mL(y), \mL'(t_1'+q))\le \veps\}$$ and
$$Q_2=\{ q\in [a',b'] \mid \exists ~y\ge t_2 \text{ with }  \ds(\mL(y), \mL'(t_1'+q))\le \veps\}.$$
Thus $Q_1, Q_2$ are nonempty closed subsets of $[a',b']$ with $Q_1\cup Q_2=[a',b']$. Since $[a',b']$ is connected, it follows that there exists some $q\in Q_1\cap Q_2$. For this $q$ there exist $y_1\le t_1$ and $y_2\ge t_2$ such that $\ds(\mL(y_1), \mL'(t_1'+q))\le \veps$ and $ \ds(\mL(y_2), \mL'(t_1'+q))\le \veps$. Hence
\[
R-2\veps\le |t_2-t_1|\le |y_2-y_1|\le \ds(\mL(y_1),\mL(y_2))\le 2\veps,
\]
which contradicts our assumption that $R>4\veps$.  Thus Case 3 cannot occur.

We have proved that $[\veps(3+2c_0), R-\veps(3+2c_0)]\subseteq Z$. By $(\ddag\ddag)$ for every $s'\in  [\veps(3+2c_0), R-\veps(3+2c_0)]$ we have $\ds(\mL'(t_1'+s'),\mL(t_1+s'))\le 2\veps(1+c_0)$.

Thus, with $c=3+2c_0$, we see that $\mL'|_{[t_1',t_2']}$ has a segment of length $R-c\veps$ that $c\veps$-fellow travels a subsegment of length $R-c\veps$ of $\mL$, as required.
\end{proof}

We conclude this subsection with the following lemma, which is our main tool for showing that geodesics converge. The reader should note the very important hypothesis that $\varphi$ has a lone axis.

\begin{lem} \label{lem:getting_close}
Suppose $\vphi \in \out$ is fully irreducible with lone axis $\mL$, that $x$ is a point on $\mL$, that $\rho \ge0$, and that $D_i$ is a sequence of positive numbers with $D_i \to \infty$. Suppose further that $(\mL_i)_{i\ge 0}$ is a sequence of simple periodic fold lines determined by expanding irreducible train track maps and satisfying that $\mL_i$ $\rho$--fellow travels $\mL$  along a length $D_i$--interval centered at $x$. Then $\mL_i \to \mL$ in $FL_r$.
\end{lem}

\begin{proof}
We prove convergence by showing that any subsequence $\{\mL_{i_j}\}_{j\in\N}$ of $\{\mL_i\}$ has a subsequence 
converging to $\mL$. Since $(\cv,d_{sym})$ is proper, each $B(D_{i_j},\rho)$ is compact, so that $\{\mL_{i_j}\}_{j}$ has finer and finer subsequences $\{\mL_{i_k}\}_{k}$ converging pointwise in the respective $B(D_{i_j},\rho)$.
 This allows us, by diagonalization, to take a subsequence of $\{\mL_{i_j}\}_{j}$ converging to a fold line $\mL'$ in $FL_r$. (That the limit $\mL'$ is indeed a folding line comes from \cite[Lemma 7.3]{br}.) By the assumptions on the $\mL_i$, we know that $\mL'$ must be a fold line in $N(\mL,\rho)$.  Moreover, by Remark~\ref{r:sg}, the $\mL_i$ are geodesics and therefore $\mL'$  is a geodesic as well. Then Lemma \ref{lem:closetoaxes} implies that $\mL'$ shares endpoints with $\mL$. We conclude that $\mL'$ is in the axis bundle of $\vphi$. Since $\vphi$ has only a lone axis, $\mL' =\mL$ and the proof is complete.
\qedhere
\end{proof}

\vskip10pt

\subsection{Principal and triangular automorphisms}\label{ss:Principal}

Our main theorem can be informally stated as saying that a random element of $\Out(F_r)$ is triangular. We explain this terminology here.

 Let $\Delta_r$ denote the graph consisting of the disjoint union of $2r-3$ triangles. As in \cite{stablestrata}, we call an ageometric fully irreducible $\vphi\in\out$ with $\iwp\cong\Delta_r$ \emph{principal}.

 By the main theorem of \cite{loneaxes}, principal outer automorphisms are lone axis outer automorphisms. We say their axis is in the \emph{principal stratum} of $Fl_r$.
 If for an ageometric fully irreducible $\vphi\in\out$ the graph $\iwp$ is a disjoint union of $\le 2r-3$ triangles, we call $\vphi$ \emph{triangular} (called \emph{principal basin} in \cite{stablestrata}).

\begin{rk}\label{rk:lone1}
In the present paper we only defined the notion of $\iwp$ for an ageometric fully irreducible $\vphi\in\out$. Thus for us principal and triangular fully irreducible outer automorphisms are ageometric by definition. However, in \cite{hm11}, Handel and Mosher defined $\iwp$ in the more general context of any nongeometric fully irreducible $\vphi\in\out$.  If $\vphi\in\out$ is a nongeometric fully irreducible outer automorphism with $\iwp$  being a disjoint union of $\le 2r-3$ triangles, then for the rotationless index of $\vphi$ we have $i(\vphi)<1-r$ and therefore $\vphi$ is ageometric in this case. Thus we could have omitted the ageometric assumption in the definition of principal and replaced it by ``nongeometric", and the ageometric property for principal and triangular fully irreducible outer automorphisms would still have been automatically satisfied in that case.
\end{rk}


It is proved in \cite{stablestrata} that principal outer automorphisms exist in each rank:

\begin{prop}[\cite{stablestrata} Example 6.1]\label{p:PrincipalExist}
For each $r\ge 3$, there exists a principal $\vphi\in\out$.
\end{prop}

Before proceeding, we record a  key consequence of being a triangular outer automorphism.

\begin{prop}\label{p:trivalent}
Let $\vphi\in\out$ be a triangular fully irreducible outer automorphism. Then every branch point in $T_+^{\vphi}$ is trivalent.

\end{prop}
\begin{proof}
By replacing $\vphi$ by a positive power, we may assume that $\vphi$ is rotationless (note that this operation does not change $T_+^{\vphi}$ or $IW(\vphi)$; in particular the outer automorphism remains triangular).
Recall that all triangular outer automorphisms are ageometric fully irreducible. As noted before, by \cite[Theorem~3.2]{bf94} every ageometric fully irreducible outer automorphism admits a train track representative without PNPs.
Thus we can represent $\vphi$ by a rotationless expanding irreducible train track map $g:\G\to\G$ with no PNPs.

Let $b\in T_+^{\vphi}$ be a branch point. Then, by Proposition~\ref{p:br}(2), we have $b=g_\infty(\tilde x)$ for a principal vertex $\tilde x\in \tilde G$ that is the center of a principal lift $\tilde g$ of $g$. Proposition~\ref{p:br}(1) then implies that $g_\infty$ induces a bijection between the set of directions at $b$ in $T_+^{\vphi}$ and the set of $\tilde g$-fixed directions in $\tilde G$. If $\tilde x$ is a lift of the (principal vertex) $x\in \G$ then $\widetilde{SW}(\tilde x)\cong SW(g,x)$ and there is a bijection between the set of $\tilde g$-fixed directions at $\tilde x$ and the set of $g$-fixed directions at $x$. By the assumptions on $\vphi$,
\[
IW(\vphi)=\sqcup SW(g,v)
\]
where the union is taken over all principal vertices $v$ of $\G$. Thus $SW(g,x)$ is one of the triangular components of $IW(\vphi)$, which implies that $\widetilde{SW}(\tilde x)$ has three vertices, and therefore $T_+^{\vphi}$  has three directions at $b$. It follows that $b$ is trivalent, as required.
\end{proof}

The following is a consequence of \cite[Theorem A]{stablestrata}, see Corollary~5.4 in \cite{stablestrata}.

\begin{prop}[\cite{stablestrata} Theorem A]\label{t:StableStrata}\cite[Corollary~5.4]{stablestrata}
Suppose that $r\ge 3$ and $\mL$ is the axis in $\cv$ of a principal $\vphi\in\out$. Then there exist constants $R,\veps>0$ so that for any simple periodic fold line $\mL'\in B(\mL,R,\veps)$, we have that $\mL'$ is an axis of a triangular fully irreducible $\vphi'\in\out$.
\end{prop}

The following is a strengthening of Proposition \ref{t:StableStrata} using Lemma \ref{lem:getting_close}.

\begin{cor}\label{cor:close_enough}
Let  $\vphi \in \out$ be a principal (and hence lone axis) fully irreducible outer automorphism with axis $\mL$ in $\cv$. Let $\rho\ge 0$ be arbitrary.  Then there exists some value $R\ge 1$ such that, whenever $\psi$ is a fully irreducible outer automorphism with an axis $\mL'$ satisfying that $\mL'\in B(\mL,R,\rho)$, we have that $\psi$ is a triangular fully irreducible outer automorphism.
\end{cor}

\begin{proof}
By Proposition~\ref{t:StableStrata}, there exist values $\veps\ge 0$ and $R'\ge 1$ such that whenever a fully irreducible $\psi\in\out$ has an axis $\mL'\in B(\mL,R',\veps)$, we have that $\psi$ is triangular.

We now argue by contradiction. Suppose the conclusion of the corollary fails. Then there exists a sequence of positive real numbers $R_i\to\infty$ as $i\to\infty$ and a sequence of fully irreducible $\psi_i\in\out$ with axes $\mL_i\in B(\mL,R_i,\rho)$ such that for every $i\ge 1$, we have that $\psi_i$ is not triangular.

Since $\mL$ is $\vphi$-periodic and a period of $\vphi$ is a compact interval, by choosing a specific point $x\in \mL$ and replacing each $\mL_i$ by some translate by a power of $\vphi$ (which corresponds to replacing $\psi_i$ by its conjugate by that power of $\vphi$), we may assume that each $\mL_i$ $\rho$--fellow travels $\mL$
along a length $D_i$--interval centered at $x$, where $\lim_{i\to\infty} D_i=\infty$. Then, by Lemma~\ref{lem:getting_close}, we have $\mL_i\to \mL$. This implies that for a sufficiently large index $i_0\ge 1$ we have that $\mL_{i_0}\in B(\mL,R',\veps)$, and therefore $\psi_{i_0}$ is a triangular fully irreducible outer automorphism, yielding a contradiction.
\end{proof}

\vskip10pt

\section{The free factor graph $\FF$} \label{sec:ff}
The \emph{free factor graph} $\FF$ for $F_r$ is the graph having a vertex for each conjugacy class of proper nontrivial free factors of $F_r$. When either $A \subsetneq B$ or $B \subsetneq A$ an edge connects the vertices $[A]$, $[B]$. The free factor graph $\FF$ was proven to be Gromov hyperbolic by Bestvina and Feighn in \cite{bf11} (this is also implied by \cite{kr14} applied to \cite{hm13free}).

Recall from Reynolds \cite{ReynoldsArational} that, for
 a point $T\in \partial\cv$, a proper free factor $A$ of $F_r$ is a \emph{reducing factor} for $T$ if there exists an $A$-invariant subtree $T'$ in $T$ such that the action of $A$ on $T'$ has a dense orbit.
In particular, if $A$ is a free factor that acts elliptically on $T$ fixing a point $p$ in $T$, then $A$ is a reducing factor, by taking $T'=\{p\}$. We denote the set of reducing factors of $T$ by $R(T)$.
The tree $T$ is \emph{arational} if it has no reducing free factors.

We let $\AT$ denote the set of \emph{arational trees} $T\in\partial\cv$ (see \cite{ReynoldsArational}). Bestvina and Reynolds \cite[Theorem 1.1]{br}, and independently Hamenst\"adt \cite{h12}, proved that $\partial\FF$ is homeomorphic to equivalence classes of arational trees. This is more precisely explained in what follows.

\vskip10pt

\subsection{The projection map $\pi\colon\cv\to \FF$}\label{ss:Projection}

\begin{df}[Projection map $\pi\colon\cv\to \FF$]\label{d:Projection}
We let $\pi\colon\cv\to \FF$ denote the \emph{projection map} defined by sending $T$ to the collection of conjugacy classes of proper free factors that act elliptically on a simplicial tree obtained by equivariantly collapsing certain edges in $T$ to points. In terms of the graph description of $\cv$, $\pi(T\slash\F_r)$ is sent to the collection of conjugacy classes of nontrivial free factors that arise as the fundamental group of a proper subgraph of $T\slash\F_r$. We will use extensively the observation that all points in the same open simplex of $\cv$ have the same projection under $\pi$.

Using the main results from Bestvina--Reynolds \cite{br},
we follow \cite[Theorem 5.5]{DahmaniHorbez} and extend $\pi$ to a map, which we also call $\pi$ (called $\psi$ in \cite{DahmaniHorbez}), $\pi\colon \uos \to \FF\cup\partial_{\infty}\FF$ satisfying: There exists a value $\eta\in\R$ so that
\begin{itemize}
  \item for each $T\in\AT$ and sequence $\{T_n\in \cv\}$ converging to $T$, the sequence $\{\pi(T_n)\}$ converges to $\pi(T) \in \partial_\infty \FF$ and
  \item for each $T\in\ol\cv \backslash \AT$ and greedy fold line $\gamma\from\R_{+}\to\cv$, if $\gamma(t)$ converges to $T$ as $t\to +\infty$ in $\uos$, then there exists an $N\in\R_{+}$ so that $d(\pi(\gamma(t)),\pi(T))\le \eta$ for each $t>N$.
\end{itemize}

As in \cite[pg. 115]{bf11}, for $\Gamma,\Gamma'\in\cv$, we define
\begin{equation}\label{FreeFactorDistance}
  d_{\FF}(\Gamma,\Gamma'):= \sup\{d_{\FF}(A,A')\mid A\in\pi(\Gamma), A'\in\pi(\Gamma')\}.
\end{equation}
\end{df}

With these definitions, $\pi\colon\cv\to \FF$ is $\out$-equivariant and coarsely $L_{\cv}$--Lipschitz \cite[Proposition 9.3]{bf11}, for some constant $L_{\cv}$ depending only on the rank of the fixed free group. In the course of establishing hyperbolicity of $\FF$, Bestvina--Feighn also showed that $\pi$-images in $\FF$ of geodesics in $\cv$ are uniform unparametrized quasi-geodesics \cite[Corollary 6.5]{bf11}. That is, they can be reparameterized to be $Q$--quasigeodesics for some $Q\ge1$ depending only on the rank of $F_r$.

It follows from the \cite[Theorem 5.5]{DahmaniHorbez} construction and \cite[Lemma 7.1]{br} that if $T \in \partial\cv$ is not arational, then $\pi(T)$ is by definition the set of reducing factors $R(T)$ of $T$. By work of Bestvina--Reynolds \cite{br}, there is a constant $\mathfrak{d} \ge 0$ such that if $S,T \in \uos$ with $R(T) \cap R(S) \neq \emptyset$,
\begin{equation}\label{e:mathfrakd}
d_{\FF}(R(T),R(S)) \le \mathfrak{d}.
\end{equation}

The constants $L_{\cv}$, $Q$, and $\mathfrak{d}$ established here will be important in this paper.

\vskip5pt

\vskip10pt

\subsection{Nearest-point projection and ${\Pr}_\gamma$}

Recall from \S \ref{ss:msnotation} the definition of a Gromov product and that, since $\FF$ is $\delta$-hyperbolic, given points $x,y,z\in\FF$, we have that
\begin{equation}\label{e:2delta}
|d(x,\mathbf{n}_{[x,y]}(z))-(x|y)_z|<2\delta,
\end{equation}
where $\mathbf{n}_{[x,y]}(z)$ denotes the nearest-point projection of $z$ to the geodesic segment $[x,y]$ connecting $x$ and $y$ in $\FF$ and $\delta$ is the hyperbolicity constant of $\FF$.

Suppose that $\gamma\from I \to \cv$ is a greedy folding path. As in 
\cite{dt17}, we use $\mathbf{n}_{\pi\circ\gamma}$ to denote the nearest-point projection to the image of $\pi\circ\gamma$ in $\FF$.

With the knowledge that a greedy folding path is a standard geodesic, the following is a direct consequence of \cite[Lemma 4.2]{dt17}. Except for where references are given, the lemma will be the extent of what we need to know about the projection function $\Pr_{\gamma}\from \cv\to\gamma(I)$ of \cite[\S 6]{bf11}. Note that \cite[\S 4.1]{dt17} is also dedicated to explaining the function. We need only the following lemma:

\begin{lem}[\cite{dt17} Lemma 4.2]\label{dt_4.2} There exists a constant $\mathfrak{c}\ge 0$, depending only on the rank $r$ of $F_r$, so that for any $T\in\cv$ and a greedy folding path $\gamma\from I \to\cv$, we have that
\[
d_{\FF}(\pi({\Pr}_\gamma (T)), \mathbf{n}_{\pi\circ\gamma}(\pi(T)))\le \mathfrak{c}.
\]
\end{lem}

For this reason, we will at times in statements replace $\pi(Pr_{\gamma}(T))$ with $\mathbf{n}_{\pi\circ\gamma}(\pi(T))$ or vice versa.


The following lemma, Lemma \ref{DH_5.11}, is indirectly stated in the proof of \cite[Proposition 8.5]{br} and then directly stated as \cite[Proposition 5.11]{DahmaniHorbez}.

\begin{lem}[\cite{DahmaniHorbez} Proposition 5.11]\label{DH_5.11}
There exists a value $M>0$ so that for all sufficiently large $B>0$, all $y_0,y_1 \in \cv$, all trees $T \in \uos$, and all sequences $(T_n)$ in $\out$ converging to $T$ we have: If $(\pi(T_n))\mid \pi(y_1))_{\pi(y_0)} \ge B$ for all $n\ge 0$, then $(\pi(T))\mid \pi(y_1))_{\pi(y_0)} \ge B - M$.
\end{lem}

\section{The bounded geodesic image property for fully irreducible axes}{\label{s:Section}}

\vskip 2pt

In this section, we show that greedy folding axes of fully irreducible outer automorphism have the bounded geodesic image property in $\cv$. Informally, this means that geodesics which are far from the axis have bounded diameter projection to the axis in $\cv$. For honest (symmetric) metric spaces, the bounded geodesic image property is equivalent to the strong contracting property \cite{arzhantseva2015growth} (strong contracting was established in \cite{a08} for axes in $\cv$). However, for the asymmetric metric spaces the contracting property does not appear to directly imply the bounded geodesic image property and additional, situation specific arguments, are required.  Here, we rely on recent work of Dahmani and Horbez \cite{DahmaniHorbez}.

\begin{df}[Bounded geodesic image property]\label{d:Definition}
A greedy folding path $\gamma$ in $\os$ has the \emph{bounded geodesic image property} if there exist constants $C_1$ and $C_2$ so that if points $X, Y \in \cv$ satisfy $d_{sym}(\Pr_\gamma(X), \Pr_\gamma(Y))\ge C_1$ and satisfy $\Pr_\gamma(X)=\gamma(t_1)$, $\Pr_\gamma(Y)=\gamma(t_2)$ for some $t_1<t_2$,  then any geodesic segment $[X,Y]$ between them contains a subsegment $[Z_1,Z_2]$ such that
\[
d_{sym}(Z_1, {\Pr}_{\gamma} (X)) \le C_2 \quad \text{and} \quad d_{sym}(Z_2, {\Pr}_\gamma(Y)) \le C_2.
\]
\end{df}




\subsection{Progressing and contracting geodesics}\label{ss:Progressing}

This subsection contains the technical results we need from Dahmani and Horbez \cite{DahmaniHorbez}. First, the space $\ue$ of \emph{uniquely ergonometric} trees is the subspace of $\partial \cv$ consisting of trees $T$ with $\pi^{-1}(\pi(T)) = T$.

%

\vskip1pt

We begin with a definition from \cite{DahmaniHorbez} (Definition 5.13). While it involves \emph{(full) recurrence}, the actual definition of full recurrence will not be necessary for our use. Please notice that we have slightly reformulated the \cite{DahmaniHorbez} definition using Lemma \ref{dt_4.2}.

\begin{df}[$(C\colon A_0,B_0,C_0)$-progressing]\label{d:progressing}
Suppose that $(T,T')\in\ue\times\ue\backslash\Delta$, that $\gamma \colon\R\to\cv$ is a greedy fold line from $T$ to $T'$, and that $S$ is a point on $\gamma$. Suppose further that $C,A_0,B_0,C_0>0$. Then $\gamma$ is said to be \emph{$(C\colon A_0,B_0,C_0)$-progressing at $S$} if the pair $(\gamma, S)$ satisfies:

If we have
\begin{itemize}
\item a tree $\widetilde S \in \gamma$  to the right of $S$ along $\gamma$ with $d_{\FF}(S,\widetilde S) \le A_0$, and
\item a tree $R \in \cv$ such that $\Pr_\gamma(R)$ is to the right of $S$ and $d_{\FF}(\widetilde Q, \Pr_\gamma(R))\ge B_0$ for any choice of $\widetilde Q$ satisfying the conditions of $\widetilde S$ above, and
\item a greedy folding path $\gamma' \colon [a,b] \to \cv$ from $S'$ to $R$, where $S'$ is either equal to $\widetilde{S}$ or is in the same simplex as $\widetilde{S}$ and is fully recurrent with respect to $R$, and
\item a value $c\in[a,b]$ so that, in the symmetrized metric, $\diam (\gamma'([a,c])\ge 3$.
\end{itemize}
Then $\diam_{\FF}(\pi\circ\gamma'([a,c])\ge C_0$.
\end{df}

Recall that $Q\ge0$ is the constant such that all $\dL$--geodesics in $\cv$ map to unparameterized $Q$--quasigeodesics under the map $\pi \colon \cv \to \FF$. As in the paragraph preceding \cite[Definition 3.3]{DahmaniHorbez}, we fix $\kappa \ge 0$ to be as in the conclusion of Proposition \ref{prop:dh}, where $Q$ is the constant fixed here.

\begin{df}[$(B,D)$-contracting]\label{d:contracting}
Suppose that $\gamma$ is a geodesic in $\cv$ and $S$ is a point on $\gamma$. Then we say that $\gamma$ is \emph{$(B,D)$-contracting at $S$} if the following holds:

Let $a\in \R$ be such that $\gamma(a)=S$ and let
$b:=\inf\{x\in\R\mid \text{diam}_{\FF}(\gamma|_{[a,x]})\ge B\}$.
Then for all geodesics $\gamma'\in\cv$, if $\pi(\gamma')$ crosses $\pi(\gamma|_{[a,b]})$ up to distance $\kappa$,
 then there exists an $a' \in \R$ so that $d_{\cv}(\gamma'(a')), S) \le D$.
\end{df}

We need the following result of Dahmani--Horbez. Informally, it states that progressing implies contracting.

\begin{prop}[\cite{DahmaniHorbez} Proposition 5.17] \label{prop:progress_contract}
There exist constants $\alpha_0,\beta_0,\gamma_0$ so that for each triple $(A_0,B_0,C_0)$ with $A_0 \ge \alpha_0$, $B_0 \ge \beta_0$, and $C_0 \ge \gamma_0$, there exists a $B >0$ satisfying:

For each $C >0$,  there exists a $D>0$ such that for all $(T,T') \in \ue \times \ue \setminus \Delta$, all greedy folding lines $\gamma \colon \R \to \cv$ from $T$ to $T'$, and all $S \in \mathrm{im}(\gamma)$, if $\gamma$ is $(C; A_0, B_0, C_0)$--progressing at $S$, then $\gamma$ is $(B,D)$--contracting at $S$.
\end{prop}


\vskip5pt

\subsection{The bounded geodesic image property for fully irreducible axes}{\label{ss:BGIPforFIs}}

In this section we will prove Theorem~\ref{t:BGIP}, establishing that a greedy folding axis in $\cv$ of a fully irreducible element of $\out$ has the bounded geodesic image property.


We will need the following lemma, Lemma \ref{l:progressing}. This lemma could be ascertained from a combination of results in \cite{DahmaniHorbez}, but we give a more direct proof in our circumstance here.


\begin{lem}\label{l:progressing}
Suppose $\gamma$ is a greedy fold axis of a fully irreducible element of $\out$. Fix $A_0,B_0,C_0 \ge 0$ sufficiently large. Then for each $S \in \gamma$ there is a $C\ge 0$ such that $\gamma$ is $(C:A_0,B_0,C_0)$ progressing at $S$.
\end{lem}

\begin{proof}[Proof of Lemma \ref{l:progressing}]
Suppose for the sake of contradiction that we have:
\begin{enumerate}
\item a sequence of $\widetilde S_n \in \gamma$ each lying to the right of $S$ along $\gamma$ and satisfying $d_{\FF}(S,\widetilde S_n) \le A_0$,
\item a sequence $R_n \in \cv$ such that $\Pr_\gamma(R_n)$ is to the right of $\widetilde S_n$ and $d_{\FF}(\widetilde S_n, \Pr_\gamma(R_n))\ge B_0$,
\item a sequence of greedy fold paths $\gamma_n' \colon [a,b] \to \cv$ from $S_n'$ to $R_n$, where $S_n'$ is either equal to $\widetilde{S}_n$ or is in the same simplex as $\widetilde{S}_n$ and is fully recurrent with respect to $R_n$
\end{enumerate}
so that the greedy fold paths $\gamma_n'$ further satisfy: if $\sigma_n$ is the smallest initial segment $\gamma_n'$ whose endpoints have symmetric distance at least $n$, then $\diam_{\FF}(\sigma_n) \le C_0$.

Note that $\diam_{\cv} (\sigma_n) \to \infty$ as $n \to \infty$, even though their projections to $\FF$ remain bounded.

We pass to a subsequence so that the $\{S_n'\}$ converge in $\ol\cv$ to an $F_r$-tree $S'$. As in \cite[Lemma 5.20]{DahmaniHorbez}, we prove:

\begin{cl}\label{claim1} $S'\in\cv$ (as opposed to $\in\partial\cv$).
\end{cl}
\begin{proof}[Proof of Claim]
Observe that the $\widetilde S_n$ vary in a compact subset of $\gamma$ since $\gamma$ projects to a quasigeodesic in $\FF$. Hence, we can pass to a subsequence so that they are all in the same simplex -- hence so are the $S'_n$. Now if $S'\in\partial \cv$, then by \cite[Lemma 7.3]{br} each element of $F_r$ acting elliptically on $S'$ also acts elliptically on $R$. Since the trees $S_n'$ live in the same open simplex in $\cv$, this means that $S'$ has a primitive elliptic element $a$, where $a$ comes from a closed loop in the underlying graph of the open simplex containing all the $S_n'$. (In particular, it is in the subgraph whose edge-lengths shrink to zero as $n\to\infty$.)
Then $\langle a\rangle$ is a common reducing factor for the trees $R$ and $S'$ and, by the definition of the map $\pi$, we have that, coarsely speaking, $\pi(R)= \pi(S)= \langle a\rangle$. The choice of $a$ also implies that $\pi(S_n')=\langle a\rangle$.

Let $M$ be the constant provided by Lemma \ref{DH_5.11}. Define $C$ to be the constant replacing the $2\delta$ in Equation \ref{e:2delta} if $[x,y]$ were instead a quasi-geodesic. And let $\mathfrak{c}$ be the constant of Lemma \ref{dt_4.2}. Assume now that $B_0\ge 10M+C+\mathfrak{c}.$
We reach a contradiction by applying Lemma \ref{DH_5.11} as follows.

Since all $S_n'$ live in the same open simplex, we have $\pi(S_n')=\pi(S_1')$ for all $n\ge 1$. So we insert into Lemma \ref{DH_5.11}: $y_0=S_1'\in \cv$, and $R_n=T_n$, and $y_1=\gamma(t_0)$ for some sufficiently large $t_0$. Namely, we choose $t_0$ large enough so that $\pi(\gamma(t_0))$ is further than $10M$ to the right of $\pi(S_1')$ along $\pi(\gamma)$.

Recall that, by assumption, in the factor graph $\FF$, we have that $R_n$ projects farther than $B_0$ to the right of $S_1'$ along $\pi(\gamma)$. By Equation \ref{e:2delta} and Lemma \ref{dt_4.2} this implies that
$$(\pi(R_n)\mid \pi(\gamma(t_0)))_{\pi(S'_1)} \ge B_0-C-\mathfrak{c}.$$
Lemma \ref{DH_5.11} now implies that
$$(\pi(R) \mid \pi \circ \gamma(t_0)))_{\pi(S'_1)} \ge B_0 - M - C - \mathfrak{c} \ge 9M.$$
This contradicts the fact that $\pi(R)=\pi(S)=\pi(S_1')=\langle a\rangle$.

Hence, $S' \in \cv$, as desired.
\end{proof}

Notice that, for sufficiently large $n$, we have that $\widetilde{S_n}$, $S_n'$, and $S'$ are all in the same simplex, so $\pi(\widetilde{S_n})=\pi(S_n')=\pi(S')$. Thus, Item 1 at the start of the proof, together with \cite[Lemma 3.1]{bf11}, implies $d_{\FF}(S, S') \le A_0 + 4$.

\vskip5pt

After passing to a subsequence, we see that $R_n\to T$ in $\uos$ for some $T \in \partial CV$.
Applying \cite[Lemma 7.3]{br}, we have that after passing to a further subsequence certain initial segments of the $\gamma_n'$ converge uniformly on compact sets to a folding ray $\gamma'$. The folding ray $\gamma'$ has a limit $S \in \partial \cv$ such that $(1)$ if $T$ is arational, then $\pi(S) = \pi(T) \in \partial \FF$ and $(2)$ if $T$ is not arational, then $\pi(T)$ and $\pi(S)$ are uniformly close in $\FF$, since $T$ and $S$ share a reducing factor. In particular, $d_\FF(\pi(T),\pi(S)) \le \mathfrak{d}$, for $\mathfrak{d}$ as defined in Equation \ref{e:mathfrakd}.

\begin{cl}\label{claim2} There exists a point $z$ on $\gamma'$ such that $d_{\FF}(S, \Pr_\gamma(z))\ge C_0 + 10L_{CV}$, where $L_{CV}$ is the Lipschitz constant for $\pi\from \cv \to \FF$.
\end{cl}
\begin{proof}[Proof of Claim]
We start by assuming that $B_0>C_0 + 10K + M + 3C + 2\mathfrak{c} +\mathfrak{d}$, where still $M$ is the constant provided by Lemma \ref{DH_5.11}, and $\delta$ is the hyperbolicity constant of $\FF$, and $C$ adjusts $2\delta$ for our using quasi-geodesics, $\mathfrak{c}$ is the constant of Lemma \ref{dt_4.2}, and the definition of $\mathfrak{d}$ is as in Equation \ref{e:mathfrakd}. Let $y_0=\pi(S)=\pi(S_1')=\pi(S_n')$.

By Item 2 at the start of the proof we know $d_{\FF}(\widetilde{S_n},\Pr_{\gamma}(R_n))\ge B_0$ and thus, by Lemma \ref{dt_4.2},
$$(\pi(\widetilde{S_n}) \mid \mathbf{n}_{\pi\circ\gamma}(\pi(R_n)))_{y_0}\ge B_0-\mathfrak{c}-C.$$
Therefore, Lemma \ref{DH_5.11} implies that, for $t_0$ sufficiently large,
\begin{equation}\label{e:cl2}
(\pi(T) \mid \pi(\gamma(t_0)))_{y_0}\ge B_0-M-C-\mathfrak{c}.
\end{equation}

We consider separately two cases:
\[\text{ (1) } T \text{ is arational and (2) } T \text{ is not arational.}\]

Suppose first that we are in Case 1, i.e. that $T$ is arational, so that $\pi(T) = \pi(S)$. Then, by \cite{br}, we know that $\pi(\gamma'(t))$ converges to the point $\pi(T)$ in $\partial\FF$ as $t\to +\infty$.
Since for ideal triangles in $\delta$-hyperbolic spaces, the Gromov product converges in a lim sup sense, we have that, for any point $\pi(z)$ on $\pi(\gamma')$ sufficiently far away from the point $y_0$:
$$(\pi(z)\mid \pi(\gamma(t_0)))_{y_0}\geq B_0-M-2C-\mathfrak{c}.$$
Equation \ref{e:2delta} then says
$$d_{\FF}(y_0,\mathbf{n}_{\pi\circ\gamma}(\pi(z)))\geq B_0-M-3C-\mathfrak{c}.$$
This in turn implies that
$$d_{\FF}(y_0,{\Pr}_{\gamma}(z))\geq B_0-M-3C-2\mathfrak{c}\geq C_0+10L_{CV}.$$
Thus the claim is satisfied using any such $z$.

We now suppose that we are in Case 2, i.e. that $T$ is not arational, so that $d_\FF(\pi(T),\pi(S)) \le \mathfrak{d}$. By the second bullet point of Definition \ref{d:Projection},
we have for all $t$ sufficiently large:
$$d_{\FF}(\pi(S), \pi(\gamma'(t)))\leq \eta,$$
for some uniform constant $\eta>0$ independent of $\gamma,\gamma'$, or $T$ (see Definition \ref{d:Projection}).
Hence, for large $t$
$$d_{\FF}(\pi(T), \pi(\gamma'(t)))\leq \eta +\mathfrak{d}.$$
By taking $t$ large enough and setting $z=\gamma'(t)$, by Equation \ref{e:cl2}, we have
$$(\pi(z) \mid \pi(\gamma(t_0)))_{y_0}\geq B_0-M-2C-\mathfrak{c} - \mathfrak{d}.$$
Using a similar argument as in Case 1, we now also have that the claim is satisfied in this case with our $z=\gamma'(t)$, where $t$ is sufficiently large.
\end{proof}

We now complete the proof as in \cite[Lemma 5.21]{DahmaniHorbez}. Recall that certain initial subsegments of the $\gamma_n'$ converge uniformly on compact sets to $\gamma$. So we can choose a sequence $\{z_n\}$ in $\cv$ converging to the point $z$ of Claim \ref{claim2} and satisfying that $z_n\in\gamma_n'$ for each $n$. Now take $n$ sufficiently large so that both $d_{sym}(S_n',S')<\frac{1}{2}$ and $d_{sym}(z_n,z)<\frac{1}{2}$. Claim \ref{claim2} implies that $d_{\FF}(S, \Pr_\gamma(z_n))\ge C_0 + 3L_{CV}$.
Observe that, since $d_{sym}(z_n,z)<\frac{1}{2}$ and $S'$ is fixed, $d_{sym}(z_n,S')$ is bounded also, and hence $d_{sym}(z_n,S_n')$ is bounded. This implies that eventually the $z_n$ lie on the segments $\sigma_n$. However, $\diam_{\FF}(\sigma_n)\le C_0$ for all $n$ and $y_0 \subset\pi(\sigma_n)$ for all $n$, so that $d_{\FF}(S, \Pr_\gamma(z_n))\le C_0$. This contradicts that $d_{\FF}(S, \Pr_\gamma(z_n))\ge C_0 + 3L_{CV}$.
\end{proof}

\begin{thm}\label{t:BGIP}
Let $\gamma$ be a greedy folding axis in $\cv$ of a fully irreducible outer automorphism $\vphi\in\out$. Then $\gamma$ has the bounded geodesic image property.
\end{thm}

\begin{proof}
By combining Lemma \ref{l:progressing} with Proposition \ref{prop:progress_contract}, we see that for each $S$ in the image of $\gamma$ there exist constants $B_1,D_1 \ge 0$ so that $\gamma$ will be $(B_1, D_1)$-contracting at $S$. It follows that $\gamma$ is $(B_1, D_1)$-contracting at each translate $\vphi^i(S)$.
If necessary, replace $\varphi$ with a positive power so that $d_\FF(S,\varphi(S)) \ge B_1$.

First, we refer the reader back to \S \ref{sec:ff} for properties of $\FF$ and add the following: In a $\delta$-hyperbolic metric space $X$, every quasigeodesic has the bounded geodesic image property. More precisely, for any $\delta\ge 0$ and $Q \ge 1 $ there exist constants $E,E_1>0$ with the following property. For any $x,y$ in a $\delta$-hyperbolic metric space $X$ and a $Q$-quasigeodesic $\alpha$ in $X$, if $\mathbf{n}_\alpha(y)$ occurs at least distance $E$ to the right of $\mathbf{n}_\alpha(x)$ along $\alpha$, then any $Q$--quasigeodesic $[x,y]$ contains a subsegment $[x',y']$ such that
$d(x',\mathbf{n}_\alpha(x))\le E_1$, and $d(y',\mathbf{n}_\alpha(y))\le E_1$, and $[x',y']$ has Hausdorff distance at most $E_1$ from $[\mathbf{n}_\alpha(x)), \mathbf{n}_\alpha(y))]$.

We set $l = \max\{\ds(S, \vphi (S)), d_{\FF}(S, \vphi(S)) \}$ and recall that $\mathfrak{c}$ is the constant of Lemma \ref{dt_4.2}. Note that $l \ge B_1$.

Now define $C_1 = K(10Kl + 2 \mathfrak{c} + E + K)$ and suppose that we have $X,Y\in CV$ such that  $\ds(\Pr_\gamma(X), \Pr_\gamma(Y)) \ge C_1$ with the projection of $Y$ to $\gamma$ occurring to the right of the projection of $X$ along the orientation of $\gamma$. Let $[X,Y]$ be any geodesic from $X$ to $Y$ in $\cv$.

We know that $\pi(\gamma)$ is a $K$-parameterized quasigeodesic in $\FF$, and that $\pi(\Pr_\gamma(X))$ and $\pi(\Pr_\gamma(Y))$ are within $\mathfrak{c}$ of the nearest point projections of $\pi(X)$ and $\pi(Y)$ to $\pi(\gamma)$. By the choice of $C_1$, this implies that $$d_{\FF}(\pi(\Pr_\gamma(X)), \pi(\Pr_\gamma(Y)))\ge 10Kl +E +2\mathfrak{c}$$
and so
$$d_{\FF}(\mathbf{n}_{\pi(\gamma)}(\pi(X)), \mathbf{n}_{\pi(\gamma)}(\pi(Y))) \ge 10Kl + E.$$


By the bounded geodesic image property for $\pi(\gamma)$, there exists a subsegment $[x',y']$ of the (unparameterized) $Q$--quasigeodesic $\pi([X,Y])$ in $\FF$ such that
\[
d(x', \mathbf{n}_{\pi(\gamma)}(\pi(X))\le E_1 \text{ and } d(y', \mathbf{n}_{\pi(\gamma)}(\pi(Y))\le E_1.
\]


We take $\pi(S_x)$ to be the translate $\vphi^i(\pi(S))$ occurring immediately to the right of $\mathbf{n}_{\pi(\gamma)}(\pi(X))$ along $\pi(\gamma)$. Note that $d_\FF(\pi(S_x) , \mathbf{n}_{\pi(\gamma)}(\pi(X))) \le l K$. Similarly, we let $\pi(S_y)$ be the translate  $\vphi^i(\pi(S))$ which is the \emph{second to last} translate of $\pi(S)$ before $\mathbf{n}_{\pi(\gamma)}(\pi(Y))$ along $\pi(\gamma)$. Then $d_\FF(\pi(S_y) , \mathbf{n}_{\pi(\gamma)}(\pi(Y))) \le 2 lK$. Our selection of $S_x$ and $S_y$ guarantees that we may apply the bounded geodesic image property at both of these points.


We first apply the $(B_1,D_1)$-contraction property for $S_x$, using the segment of $[x',y']$ that crosses a subsegment of $\pi(\gamma)$ starting at $\pi(S_x)$ with length $B_1$, and find a point $Z_1$ on the corresponding segment $J_1$ of $[X,Y]$ such that $\ds(Z_1,S_x)\le D_1$. Then we apply the $(B_1,D_1)$-contraction property for $S_y$, using the segment of $[x',y']$ that crosses a subsegment of $\pi(\gamma)$ starting at $\pi(S_y)$ with length $B_1$, and find a point $Z_2$ on the corresponding segment $J_2$ of $[X,Y]$ such that $\ds(Z_2,S_y)\le D_1$. By construction, $J_1$ occurs before $J_2$ in $[X,Y]$ and therefore $Z_1$ comes before $Z_2$ along $[X,Y]$.   Since $\gamma$ is thick, there is a $C_2$ such that $\ds(Z_1,S_x)$ and $\ds(Z_2, S_y)$ are both less then $C_2$.
\end{proof}

\begin{corollary}\label{cor:BGI}
Let $\gamma$ be a greedy folding axis in $\cv$ of some fully irreducible $\vphi\in \out$. Then there exist constants $\rho\ge 0$ and $c_0\ge 0$ with the following property.
Let $C_1,C_2>0$ be the bounded geodesic image property constants for $\gamma$ (where we know that this property holds for $\gamma$ by Theorem \ref{t:BGIP}). Let $R\ge 1$ be an arbitrary sufficiently large number.

Let $X,Y\in \cv$ be such that $\ds(\Pr_\gamma(X), \Pr_\gamma(Y))\ge C_1$ and such that  $\Pr_\gamma(X)=\gamma(t_1)$, $\Pr_\gamma(Y)=\gamma(t_2)$ where $t_2-t_1\ge R$.  Then any geodesic segment $[X,Y]$ contains a subsegment of length $\ge R-c_0$ that $\rho$-fellow travels a subsegment of the same length in $\gamma$.
\end{corollary}

To prove the corollary, we need the following lemma.

\begin{lem}\label{l:morse}
Let $K\ge 0$ and $C\ge 0$. Then there exists a constant $A'=A'(K,C)>0$ with the following property. Let $\gamma$ be a geodesic in $\cv$ such that $\pi\circ \gamma$ is a $K$-quasigeodesic in $\FF$.  Let $P,Q\in \gamma$ and $Z_1,Z_2\in \cv$ be such that $\ds(Z_1,P)\le C$ and $\ds(Z_2,P)\le C$. Let $[P,Q]$ denote the segment of $\gamma$ from $P$ to $Q$. Then for any geodesic segment $[Z_1,Z_2]$ in $\cv$, we have that the segments $[P,Q]$ and $[Z_1,Z_2]$ are $A$-Hausdorff close with respect to $\ds$.
\end{lem}

\begin{proof}
Consider the geodesic segments $[P,Z_1]$, $[Z_2,Q]$ and the path $\alpha=[P,Z_1]\cup [Z_1,Z_2]\cup [Z_2,Q]$. Since $\ds(Z_1,P)\le C$ and $\ds(Z_2,P)\le C$, the path $\alpha$ is a $C'$-quasigeodesic in $\cv$ for some constant $C'\ge 0$ depending only on $C$. Therefore, by  \cite[Theorem 4.1]{dt17}, for some constant $A=A(C',K)\ge 0$, the paths $\alpha$ and $[P,Q]$ are $A$-Hausdorff close with respect to $\ds$. It follows that $[P,Q]$ and $[Z_1,Z_2]$ are $(A+C)$-close, and the lemma holds with $A'=A+C$.
\qedhere
\end{proof}

We can now prove Corollary \ref{cor:BGI}.

\begin{proof}[Proof of Corollary \ref{cor:BGI}]
By the bounded geodesic image property for $\gamma$, there exists a subsegment $[Z_1,Z_2]$ in $[X,Y]$ such that $\ds(\Pr_\gamma(X), Z_1)\le C_2$ and $\ds(\Pr_\gamma(Y), Z_2)\le C_2$. Since $\gamma$ is the axis of a fully irreducible, $\pi\circ \gamma$ is a $K$-quasigeodesic in $\FF$ for some $K\ge 0$. Thus, by Lemma~\ref{l:morse}, the segments $[Z_1,Z_2]$ and $[\Pr_\gamma(X), \Pr_\gamma(Y)]=[\gamma(t_1), \gamma(t_2)]$ are $A'$-Hausdorff close with respect to $\ds$.

By the triangle inequality we have $\dL(Z_1,Z_2)\ge R-2C_2$. Then by Lemma~\ref{l:fell} the segment $[Z_1,Z_2]$ has a subsegment of length $\ge R-2C_2-cC_2$ that $cC_2$ fellow travels a segment of $\gamma$ of the same length (where $c$ is the constant depending on thickness of $\gamma$ provided by Lemma~\ref{l:fell}).
Therefore the conclusion of the corollary holds with $\rho=cC_2$ and $c_0=2C_2+cC_2$
\end{proof}

\section{Random outer automorphisms are triangular}

Let $\vphi$ be a principal fully irreducible outer automorphism (which exists by Proposition \ref{p:PrincipalExist}) and let $\gamma$ be its lone folding axis (see \S \ref{ss:principal}). Evidently, $\gamma$ is both a greedy and a simple folding path.

By Theorem \ref{t:BGIP}, there are constants $C_1, C_2 \ge0$ such that $\gamma$ is $(C_1,C_2)$--contracting. Let $\rho \ge 0$ be the constant given by Corollary \ref{cor:BGI}. With this set-up we have the following statement:

\begin{prop} \label{prop:proj_tri}
There exists a constant $R_1 \ge0$ satisfying the following: Let $\psi$ be a fully irreducible outer automorphism with simple folding axis $\gamma'$. Suppose that the $\Pr_\gamma$--projection of $\gamma'$ to $\gamma$ has diameter $\ge R_1$ and that, using the orientations on these paths, the left end of $\gamma'$ projects to the left of the right end of $\gamma'$. Then $\psi$ is triangular.
\end{prop}

\begin{proof}
The conclusion follows directly from  Corollary \ref{cor:close_enough} and Corollary \ref{cor:BGI}.
\end{proof}

Proposition~\ref{prop:proj_tri} is a key technical result that combines the full strength of the bounded geodesic image property (Corollary \ref{cor:BGI}) and of the ``stability" result for triangular and principal outer automorphisms (Corollary \ref{cor:close_enough}).

We now turn to the proof of our main results stated in the introduction. To do so, we freely use the constants collected in \S \ref{sec:ff}. To simplify notation, we denote the $\pi$--image of a point $x$ in $\cv$ by $\overline x \subset \FF$.

\begin{theore}\label{t:main}
Let $r\ge 3$ and let $\mu$ be a probability distribution on
$\Out(F_r)$ which is nonelementary and has bounded support for the action on $\FF$ and let
$(w_n)$ be the random walk determined by $\mu$. Suppose
that $\smgp$ contains $\vphi^{-1}$ for some principal fully irreducible $\varphi\in\out$. Then $w_n$ is
triangular fully irreducible outer automorphism with probability tending to
$1$ as $n \to \infty$.
\end{theore}

\begin{proof}

We work with the constants established at the beginning of this section. As before, we apply Proposition \ref{prop:MS} to the action $\Out F_r \curvearrowright \FF$ with $Q$ equal to the (unparameterized) quasigeodesic constants for the image of a geodesic in $\cv$. We set $L = L_{\cv}\cdot (R_1+2\kappa +2 \mathfrak{c}+1)$, for $R_1$ as in  Proposition \ref{prop:proj_tri} and $\mathfrak{c}$ as in Proposition \ref{dt_4.2}.

Let $\gamma(t)$ be the (lone) axis in $\cv$ of our principal fully irreducible $\vphi$. By Remark~\ref{r:RL}, for the left action of $\out$ on $\cv$, iteration of $\vphi^{-1}$ on $\gamma$ translates in the direction of $t\to\infty$ along $\gamma$.  Then $\overline\gamma(t)$ is a directed quasigeodesic axis for $\vphi^{-1}$ in $\FF$ (again with respect to the left action of $\out$).
 Thus we can apply  Proposition \ref{prop:MS} to $\overline\gamma$ since, by assumption, $\vphi^{-1}\in\smgp$.

Letting $\gamma_{w_n}$ denote a simple folding axis for $w_n$ (as $w_n$ is fully irreducible with probability tending to $1$ as $n\to \infty$), it follows by  Proposition \ref{prop:MS}  that $\overline \gamma_{w_n}$ has an $(L, \kappa)$--oriented match with $\overline \gamma$. By the hyperbolic geometry of $\FF$, this means that there is a $g_n \in \Out(F_r)$ such that the nearest point projection in $\FF$ of $g_n \cdot \overline \gamma_{w_n}$ to $\overline \gamma$ has diameter at least $L_{\cv}\cdot (R_1 +2 \mathfrak{c}+1)$. If we instead use the projection $\pi \circ \Pr_\gamma$, we find that the diameter of the image of the projection is at least $L_{\cv} \cdot (R_1+1)$ (see Proposition \ref{dt_4.2}). Further, since the match is oriented,  we have that the left end of $\overline \gamma_{w_n}$ projects to the left of its right end, along $\gamma$.
Finally, we use the fact that $\pi \colon \cv \to \FF$ is $L_{\cv}$--Lipschitz to conclude that the $\Pr_\gamma$--projection of $\gamma_{w_n}$ to $\gamma$ has diameter at least $R_1$. Hence, from Proposition~\ref{prop:proj_tri}, we conclude that $w_n$ is triangular.
\end{proof}

\stepcounter{theore}

In the following corollary $i(\vphi_n)$ is the rotationless index of $\vphi_n$ (see \S \ref{ss:indices}) and $ind_{geom} (T_+^{\vphi_n})$ is the geometric index of $\vphi_n$ (see \cite{ch12}).

\begin{coroll}\label{c:rw}
Let $r\ge 3$ and let $\mu$ be a probability distribution on
$\Out(F_r)$ which is nonelementary and has bounded support for the action on $\FF$, and let
$(w_n)$ be the random walk determined by $\mu$.  Suppose
that $\smgp$ contains $\vphi^{-1}$ for some
principal fully  irreducible $\vphi\in\out$.

Then, with
probability going to $1$ as $n \to \infty$, $w_n$ is an \emph{ageometric}
fully irreducible outer automorphism $\vphi_n$ satisfying:
\begin{enumerate}
\item each component of the ideal Whitehead graph $\iwp$ is a triangle;
\item each branchpoint of $T_+^{\vphi_n}$ is trivalent;
\item $i(\vphi_n)>1-r$ and $ind_{geom} (T_+^{\vphi_n})<2r-2$;
\item the tree $T_+^{\vphi_n}$ is nongeometric.
\end{enumerate}
\end{coroll}

\begin{proof}
Theorem~\ref{t:main} implies that, with
probability going to $1$ as $n \to \infty$, the outer automorphism $\vphi_n$ is triangular fully irreducible. Thus, in particular, $\vphi_n$ is ageometric, so that the tree  $T_+^{\vphi_n}$ is nongeometric and $i(\vphi_n)>1-r$. Also, for nongeometric trees we have $ind_{geom} (T_+^{\vphi_n})<2r-2$ (see \cite[p. 300]{ch12} for more details).
Proposition~\ref{p:trivalent} now implies that $T_+^{\vphi_n}$ is trivalent.
\end{proof}

\stepcounter{theore}

\begin{coroll}\label{c:two}
Let $r\ge 3$ and let $\mu$ be a probability distribution on
$\Out(F_r)$ which is nonelementary and has bounded support for the action on $\FF$, and let $(w_n)$ be the random walk determined by $\mu$. Suppose that $\smgp$ contains a pair
of elements $\vphi$ and $\psi$, such that both $\vphi$ and $\psi^{-1}$ are principal fully irreducible outer automorphisms.

Then, with probability approaching $1$ as
$n \to \infty$, both $w_n$ and $w_n^{-1}$ are triangular fully irreducible outer automorphisms. In particular, both fixed trees of $w_n$ are trivalent.
\end{coroll}

\begin{proof}
Let $\mathcal{T} \subset \Out(F_r)$ be the subset consisting of the triangular fully irreducible elements.
By Theorem \ref{t:main}, $\mu^{n}(\mathcal{T}) \to 1$. Recall that
$\check \mu$ denotes the reflected measure,
$\check \mu (g) = \mu (g^{-1})$, so the same theorem implies that
$\check \mu^{n}(\mathcal{T}) \to 1$. But the distribution of the
random element $w_n^{-1}$ is $\check \mu^n$, and Theorem~\ref{t:main} applies to $\check \mu$ as well. Therefore the probability
that both $w_n$ and $w_n^{-1}$ lie in $\mathcal{T}$ tends to $1$ as
$n\to\infty$.
\end{proof}


\vskip1pt
\bibliographystyle{alpha}
\bibliography{References}

\end{document}